\documentclass[12pt]{article} 
\usepackage{fullpage}
\usepackage{url}

\newcommand{\ADJUST}{{\rm ADJUST}}
\newcommand{\DFS}{{\rm DFS}}
\newcommand{\POTB}{{\rm POTB}}
\newcommand{\POTG}{{\rm POTG}}
\newcommand{\POTF}{{\rm POT}}
\newcommand{\TFREE}{{\rm THREE\hbox{\_}FREE}}
\newcommand{\GATHER}{{\rm GATHER}}
\newcommand{\INITGATHER}{{\rm INITGATHER}}
\newcommand{\GATHERDFS}{{\rm GATHER\hbox{\_}DFS}}
\newcommand{\FINDsz}{{\rm FINDsz}}
\newcommand{\FINAL}{{\rm FINAL}}
\newcommand{\FINALDFS}{{\rm FINAL\hbox{\_}DFS}}
\newcommand{\BASIC}{{\rm BASIC}}
\newcommand{\BASICDFS}{{\rm BASIC\hbox{\_}DFS}}
\newcommand{\BASICtwo}{{\rm BASIC2}}
\newcommand{\BASICDFStwo}{{\rm BASIC\hbox{\_}DFS2}}

\newcommand{\forb}{f}
\newcommand{\forbp}{f'}
\newcommand{\forbs}[1]{f_{#1}}
\newcommand{\forbps}[1]{{f'}_{#1}}
\newcommand{\alphas}[1]{\alpha_{#1}}

\newcommand{\NUM}{{\rm NUM}}

\newcommand{\FALSE}{{\rm FALSE}}

\newcommand{\TRUE}{{\rm TRUE}}

\newcommand{\usz}{{\rm usz}}
\newcommand{\sz}{{\rm sz}}
\newcommand{\SZ}{{\rm SZ}}
\newcommand{\GF}{{\rm GF}({2^n})}
\newcommand{\KD}{{\rm KD}}

\newcommand{\GOOD}{{\rm GOOD}}
\newcommand{\lastnum}{186}
\newcommand{\lastnump}{187}
\newcommand{\lastnumtwo}{250}

\newcommand{\Ld}{L_{\rm diag}}
\newcommand{\st}{\, : \, }

\newcommand{\bits}[1]{\{0,1\}^{{#1}}}
\newcommand{\bit}{\{0,1\}}
\newcommand{\nat}{{\sf N}}
\newcommand{\implies}{\Rightarrow}
\newcommand{\into}{\rightarrow}

\newcommand{\ceil}[1]{\left\lceil {#1}\right\rceil}
\newcommand{\floor}[1]{\left\lfloor{#1}\right\rfloor}

\newcommand{\union}{\cup}

\newcommand{\s}[1]{\s_{#1}}

\newcounter{savenumi}

\newtheorem{theoremfoo}{Theorem}[section] 
\newenvironment{theorem}{\pagebreak[1]\begin{theoremfoo}}{\end{theoremfoo}}

\newtheorem{lemmafoo}[theoremfoo]{Lemma}
\newenvironment{lemma}{\pagebreak[1]\begin{lemmafoo}}{\end{lemmafoo}}
\newtheorem{conjecturefoo}[theoremfoo]{Conjecture}

\newtheorem{conventionfoo}[theoremfoo]{Convention}

\newtheorem{porismfoo}[theoremfoo]{Porism}

\newtheorem{gamefoo}[theoremfoo]{Game}

\newtheorem{corollaryfoo}[theoremfoo]{Corollary}

\newtheorem{openfoo}[theoremfoo]{Open Problem}

\newtheorem{exercisefoo}{Exercise}

\newcommand{\fig}[1] 
{
 \begin{figure}
 \begin{center}
 \input{#1}
 \end{center}
 \end{figure}
}

\newtheorem{potanafoo}[theoremfoo]{Potential Analogue}

\newtheorem{notefoo}[theoremfoo]{Note}
\newenvironment{note}{\pagebreak[1]\begin{notefoo}\rm}{\end{notefoo}}

\newtheorem{notabenefoo}[theoremfoo]{Nota Bene}

\newtheorem{nttn}[theoremfoo]{Notation}

\newtheorem{empttn}[theoremfoo]{Empirical Note}

\newtheorem{examfoo}[theoremfoo]{Example}

\newtheorem{dfntn}[theoremfoo]{Def}
\newenvironment{definition}{\pagebreak[1]\begin{dfntn}\rm}{\end{dfntn}}

\newtheorem{propositionfoo}[theoremfoo]{Proposition}

\newenvironment{proof}
    {\pagebreak[1]{\narrower\noindent {\bf Proof:\quad\nopagebreak}}}{\QED}
\newenvironment{sketch}
    {\pagebreak[1]{\narrower\noindent {\bf Proof sketch:\quad\nopagebreak}}}{\QED}

\newcommand{\yyskip}{\penalty-50\vskip 5pt plus 3pt minus 2pt}
\newcommand{\blackslug}{\hbox{\hskip 1pt
        \vrule width 4pt height 8pt depth 1.5pt\hskip 1pt}}
\newcommand{\QED}{{\penalty10000\parindent 0pt\penalty10000
        \hskip 8 pt\nolinebreak\blackslug\hfill\lower 8.5pt\null}
        \par\yyskip\pagebreak[1]}

\newcommand{\BBB}{{\penalty10000\parindent 0pt\penalty10000
        \hskip 8 pt\nolinebreak\hbox{\ }\hfill\lower 8.5pt\null}
        \par\yyskip\pagebreak[1]}

\newtheorem{factfoo}[theoremfoo]{Fact}
\newenvironment{fact}{\pagebreak[1]\begin{factfoo}}{\end{factfoo}}

\newcommand{\Erdos}{Erd\"os }

\begin{document}

\title{Finding Large Sets Without Arithmetic Progressions of Length Three: An Empirical View and Survey II}

\author{
{William Gasarch}
\thanks{University of Maryland at College Park,
Department of Computer Science,
        College Park, MD\ \ 20742.
\texttt{gasarch@umd.edu}
}
\\ {\small Univ. of MD at College Park}
\and
{James Glenn}
\thanks{Yale University ,
  Department of Computer Science,
  New Haven, CT 06511.
\texttt{james.glenn@yale.edu}
}
\\ {\small Yale University}
\and
{Clyde Kruskal}
\thanks{University of Maryland at College Park,
Department of Computer Science,
        College Park, MD\ \ 20742.
\texttt{ckruskal@umd.edu}
}
\\ {\small Univ. of MD at College Park}
}

\maketitle

\begin{abstract}
There has been much work on the following question: given $n$,
how large can a subset of $\{1,\ldots,n\}$ be that has no
arithmetic progressions of length 3.
We call such sets {\it 3-free}.
Most of the work has been asymptotic.
In this paper we sketch applications of large 3-free sets,
review the literature of how to construct large 3-free sets, and present
empirical studies on how large such sets actually are.
The two main questions considered are (1) How large can a 3-free set
be when $n$ is small, and
(2) How do the methods in the literature compare to each other?
In particular, when do the ones that are asymptotically
better actually yield larger sets?
(This paper overlaps with our previous paper with the title
{\it Finding Large 3-Free Sets I: the Small $n$ Case}.) 
\end{abstract}

\tableofcontents

\section{Introduction}

This paper overlaps with our previous paper with the title
{\it Finding Large 3-Free Sets I: the Small $n$ Case}\cite{GGK-2008}.

\subsection{Historical Background}

The motivation for this paper begins with van der Waerden's theorem:

\begin{definition}
Let $[n]$ denote the set $\{1,\ldots,n\}$.
\end{definition}

\begin{definition}
A $k$-AP is an arithmetic progression of length $k$.
\end{definition}

\begin{theorem}[\cite{VDW-1927,VDW-1971,GRS-1990}]
For all $k$, for all $c$, there exists $W(k,c)$ such that
for all $c$-colorings of $[W(k,c)]$ there exists
a monochromatic $k$-AP.
\end{theorem}

The numbers $W(k,c)$ are called {\it van der Waerden numbers}.
In the original proof of van der Waerden's theorem  the
bounds on $W(k,c)$ were quite large. 
Erdos and Turan~\cite{ET-1936} wanted a different proof
which would yield smaller bounds on $W(k,c)$.
They adopted a different viewpoint:
If  $[W(k,c)]$ is  $c$-colored then the most common color
appears at least $W(k,c)/c$ times.
They thought that 
the monochromatic $k$-AP would be in the most common color.
Formally they  conjectured the following:

\smallskip

{\it For every $k\in \nat$ and $\lambda>0$ there exists
$n_0(k,\lambda)$ such that, for every $n\ge n_0(k,\lambda)$,
for every $A\subseteq [n]$, if $|A|\ge \lambda n$
then $A$ has a $k$-AP.}

\smallskip

The $k=3$ case of this conjecture was
originally proven by Roth~\cite{GRS-1990,Roth-1953,Roth-1952}
using analytic means.
The $k=4$ case was proven by
Szemeredi~\cite{Szemeredi-1974} (see also Gowers proof~\cite{Gowers-1998}) by a
combinatorial argument.
Szemeredi~\cite{Szemeredi-1975} later
proved the whole conjecture
with a much harder proof.
Furstenberg~\cite{Furstenberg-1977}  
provided a very different proof using Ergodic theory.
Gowers~\cite{Gowers-2001} 
provided an analytic proof (though it also used combinatorics) 
that yielded much smaller upper bounds
for the van der Waerden Numbers.

We are concerned with the $k=3$ case.

\begin{definition}
For $k\in\nat$, a set $A$ is {\it $k$-free} if
it does not have any arithmetic progressions of size $k$.
\end{definition}

\begin{definition}
Let $\mbox{ \it sz }(n)$ be the maximum size of a 3-free subset of $[n]$.
(`sz' stands for Szemeredi.)
\end{definition}

Roth's theorem \cite{Roth-1952} (but see also \cite{GRS-1990})  yields the upper bound

$$(\forall \lambda)(\exists n_0)(\forall n\ge n_0)[\sz(n) \le \lambda n].$$

Roth later~\cite{Roth-1953} improved this to

$$(\exists c)(\exists n_0)(\forall n\ge n_0)\bigg [\sz(n) \le \frac{cn}{\log\log n}\bigg].$$

Better results are known:
Heath-Brown~\cite{HB-1992} obtained

$$(\exists c)(\exists n_0)(\forall n\ge n_0)\bigg [sz(n) \le \frac{n}{(\log n)^c}\bigg ].$$

Szemeredi obtained $c=1/20$.
Bourgain \cite{Bourgain-1999}
has shown that, for all $\epsilon$,  $c=\frac{1}{2}-\epsilon$ works.
In the same paper he showed

$$
(\exists c)(\exists n_0)(\forall n\ge n_0)\bigg [sz(n) \le cn\sqrt{\frac{\log\log n}{\log n}}\bigg].
$$

The above discussion gives an asymptotic upper bound on $\sz(n)$.
What about lower bounds?
That is, how large can a 3-free set of $[n]$ be when $n$ is large?

The best asymptotic lower bound is Behrend's \cite{Behrend-1946} (but also
see Section~\ref{se:sphere} of this paper)
construction of a 3-free set which yields that there exist constants $c_1,c$ such that
$\sz(n) \ge c_1n^{1-c/\sqrt{\log n}}$.

Combining the above two results we have the following:
There exist constants $c_1,c_2,c$ such that
$$c_1n^{1-c/\sqrt{\log n}} \le \sz(n) \le c_2 n\sqrt{\frac{\log\log n}{\log n}}.$$

$$
(\exists c)(\exists n_0)(\forall n\ge n_0)\bigg [sz(n) \le cn\sqrt{\frac{\log\log n}{\log n}}\bigg].
$$

Our paper investigates empirical versions of these theorems (see next section for details).
Prior empirical studies have been done by 
\Erdos and Turan~\cite{ET-1936}, Wagstaff~\cite{Wagstaff-1972} and
Wroblewski~\cite{Wroblewski-Unknown}.
\Erdos and Turan~\cite{ET-1936} computed $\sz(n)$ for $1\le n\le 21$.
Wagstaff~\cite{Wagstaff-1972} computed $\sz(n)$ for $1\le n\le 52$
(he also looked at 4-free and 5-free sets).
Wroblewski~\cite{Wroblewski-Unknown}
has on his website,
in different terminology, the values of $\sz(n)$ for $1\le n\le 150$.
We compute $\sz(n)$ for $1\le n\le \lastnum$ and
get close (but not matching) upper and lower bounds for
$\lastnump \le n\le \lastnumtwo$.
We also obtain new lower bounds on $\sz$ for
three numbers.  
Since Wroblewski's website uses a different notation than
our paper we discuss the comparison in Appendix I.

\subsection{Our Results and  A Helpful Fact}

This paper has two themes:

\begin{enumerate}
\item
For small values of $n$, what is $\sz(n)$ exactly?
\item
How large does $n$ have to be before the asymptotic results
are helpful?
\end{enumerate}

In Section~\ref{se:app} we have a short summary of how 3-free sets have been used in mathematics
and computer science.
In Section~\ref{se:small} we develop new techniques to find $\sz(n)$ for small values of $n$, where
by ``small'', we mean $n\le \lastnumtwo$. We obtain the following:
\begin{enumerate}
\item
exact values of $\sz(n)$ for $1\le n\le \lastnum$; and
\item
upper and lower bounds for $\sz(n)$ for $\lastnump\le n\le \lastnumtwo$.
\end{enumerate}
In Section~\ref{se:large} we summarize several known methods for obtaining large 3-free sets of $[n]$
when $n$ is large.  The method that is best asymptotically ---
{\it the Sphere Method} --- is nonconstructive.
We present it and several variants.
Since the Sphere method is nonconstructive it might seem
impossible to code. However, we have coded it up along with the variants of it, and the
other methods discussed.  In Section~\ref{se:3free},  we present the data and discuss
what it means. Evidence suggests that, for $n\ge 10^9$,
the nonconstructive sphere methods produce larger 3-free sets than any of the
other known methods.
In Section~\ref{se:upper} we
use the proof of the first Roth's theorem, from \cite{GRS-1990},
to obtain lower bounds on $\sz(n)$ for
actual numbers $n$.
Using Roth's theorem yields better results than naive methods.
The proofs of the results mentioned above by Szemeredi, Heath-Brown, and Bourgain
may lead to even better results; however, these proofs are somewhat
difficult and may well only help for $n$ too large for a computer to handle.
Note that the quantity from Bourgain's proof,

$$\sqrt{\frac{\log\log n}{\log n}},$$

\noindent
only helps us if it is small.  Using base 2 and $n=2^{1024}$ this
quantity is $\sqrt{\frac{10}{1024}}\approx 0.1$, which is not that small.
Note that $n=2^{1024}$ is already far larger than any computer can handle
and the advantage over Roth's (first) theorem seems negligible.
Even so, it would be of some interest for someone to try.
This would entail going through (say) Bourgain's proof and tracking down
the constants.

The next fact is trivial to prove; however, since we use it throughout
the paper we need a shorthand way to refer to it:

\begin{fact}\label{fa:twoy}
Let $x < y < z$.
Then
$x,y,z$ is a 3-AP iff $x+z=2y$.
\end{fact}

\section{Applications}\label{se:app}

We sketch four applications of 3-free sets.
The first is a combinatorics problem about chess
and the other three are
applications in theoretical computer science.

\subsection{The Diagonal Queens Domination Problem}

How many queens do you need to place on an $n\times n$ chess board
so that every square is either occupied or under attack?
How many queens do you need if you insist that they are on
the main diagonal? The former problem has been studied in~\cite{Guy-1981}
and the latter in~\cite{CH-1986-queens}.
It is the diagonal problem that is connected to 3-free sets.

\begin{theorem}
Let $diag(n)$ be the minimal number of queens needed
so that they can be placed on the main  diagonal of an $n\times n$
chessboard such that every square is either occupied
or under attack. Then, for $n\ge 2$, $diag(n)=n-\sz(\ceil{n/2})$.
\end{theorem}

The theorems surveyed in this paper will show that, for large $n$,
you need `close to' $n$ queens.

\subsection{Matrix Multiplication}

It is easy to multiply two $n\times n$ matrices in $O(n^3)$ steps.
Strassen showed how to lower this to $O(n^{2.87})$ ~\cite{Strassen-1969}
(see virtually any algorithms textbook). 
The basis of this algorithm is a way to multiply two $2\times 2$ matrices
using only 7 multiplications (but 18 additions).
For a long time (1990-2012), Coopersmith and Winograd's algorithm~\cite{CW-1990}
was the best matrix multiplication algorithm; it takes $O(n^{2.36})$ steps.
The current best algorithm, running in $O(n^{2.371552})$ steps,
is by Virginia Vassilevska Williams et al.~\cite{WXXZ-2024}.
Both algorithm use large 3-free sets.

Both algorithms need 3-free sets of size $n^{1-o(1)}$ which, as we
will discuss later, are known to exist.
Unfortunately, larger 3-free sets will not lead to better
matrix multiplication algorithms.

\subsection{Application to Communication Complexity}

\begin{definition}
Let $f$ be any function from $\bits L \times \bits L\times \bits L$ to $\bit $.
\begin{enumerate}
\item
A {\it protocol } for computing $f(x,y,z)$, where Alice has $x,y$,  Bob has $x,z$,
and Carol has $y,z$,
is a procedure where they take turns broadcasting information until
they all know $f(x,y,z)$.
(This is called `the forehead model' since we can think of Alice having $z$ on her
forehead, Bob having $y$ on his forehead, and Carol having $x$ on her
forehead. Everyone can see all foreheads except his or her own.)
\item
Let $d_f(L)$ be the number of bits transmitted in the optimal deterministic protocol for $f$.
This is called the multiparty communication complexity of $f$.
(The literature usually denotes $d_f(L)$ by $d(f)$ with the $L$ being implicit.)
\end{enumerate}
\end{definition}

\begin{definition}\label{de:exn}
Let $L\in \nat$.
We view elements of $\bits L$ as $L$-bit numbers in base 2.
Let $f:\bits L \times \bits L \times \bits L \into \bit$ be defined as
$$f(x,y,z)= \cases { 1 & if $x+y+z=2^L$; \cr
                     0 & otherwise. \cr
}
$$
\end{definition}

The multiparty communication complexity of $f$ was studied by \cite{CFL-1983}
(see also \cite{KN-1997}). They used it as a way of studying branching programs.
A careful analysis of the main theorem of \cite{CFL-1983} yields the following.

\begin{theorem}\label{th:upper}
Let $f$ be the function in Definition~\ref{de:exn}.
\begin{enumerate}
\item
$$d_f(L)=O\left(\log\left( \frac{L2^L}{\sz(2^L)} \right)\right).$$
\item
We will later see that $\sz(2^L)\ge 2^{L-c\sqrt L }$ (\cite{Behrend-1946} or Section~\ref{se:sphere}
of this paper). Hence $d_f(L)=O(\sqrt L)$.
\end{enumerate}
\end{theorem}

\subsection{Linearity Testing}

One ingredient in the proofs about probabilistically checkable proofs (PCPs)
has been linear testing~\cite{ALMSS-1998,AS-1998}.
Let $\GF$ be the finite field on $2^n$ elements
(GF stands for `Galois Field').
Given a black box for a function $f:\GF  \into \bit$,
we want to test if it is linear.
One method, first suggested by~\cite{BLR-1993}, is to pick $x,y\in \GF$ at random and
see if $f(x+y)=f(x)+f(y)$. This test can be repeated to reduce
the probability of error.

We want a test that, for functions $f$ that are `far from' linear,
will make fewer queries to obtain the same error rate.
The quantity $d(f)$ (different notation from the $d_f(L)$ in the last section)
is a measure of how nonlinear $f$ is. The more nonlinear $f$ is, the
smaller $d(f)$ is (see~\cite{ST-2000-pcp,HW-2003}).

In~\cite{ST-2000-pcp} the following was suggested:
Let $G=(V,E)$ be a graph on $k$ vertices.
For every $v\in V$ pick $\alpha(v)\in \GF$ at random.
For each $(u,v)\in E$,  test if
$f(\alpha(u)+\alpha(v))=f(\alpha(u))+f(\alpha(v))$.
Note that this test makes $k$ random choices from $\GF$
and $|E|$ queries.
In~\cite{ST-2000-pcp} they showed that, using this test,
the probability of error is $\le 2^{-|E|} + d(f)$.

In~\cite{HW-2003} a graph is used that obtains
probability of error $\le 2^{-k^2-o(1)} + d(f)^{k^{1-o(1)}}$.
The graph uses 3-free sets.
It is a bipartite graph $(X,Y,E)$ such that the following happens:
\begin{itemize}
\item
There exists a partition of $X\times Y$ into $O(k)$ sets of the form $X_i\times Y_i$.
We denote these $X_1\times Y_1$, $X_2\times Y_2$, $\ldots$, $X_k\times Y_k$.
\item
For all $i$, the graph restricted to $X_i\times Y_i$ is a matching
(i.e., it is a set of edges that do not share any vertices).
\end{itemize}
This is often expressed by saying that the graph is the union of $O(k)$ induced matchings.


We reiterate the construction of such a graph from \cite{HW-2003}.
Let $A\subseteq [k] $ be a 3-free set.
Let $G(A)$ be the bipartite graph on vertex sets $U=[3k]$
and $V=[3k]$ defined as the union over all $i\in [k]$ of
$M_i=\{(a+i,a+2i) \mid a\in A\}$.
One can check that each $M_i$ is an induced matching.

\section{What Happens for Small $n$?}\label{se:small}

In this section we present several techniques for
obtaining exact values, and upper and lower bounds,
on $\sz(n)$.

Subsection 1 describes {\it The Base 3 Method} for obtaining
large (though not optimal) 3-free sets.
Subsection 2 describes {\it The Splitting Method } for obtaining
upper bounds on $\sz(n)$.
Both the Base 3 method and the Splitting Method are easy; the
rest of our methods are more difficult.
Subsection 3 describes an intelligent backtracking method
for obtaining $\sz(n)$ exactly.
It is used to obtain all of our exact results.
Subsection 4 describes how to use linear programming
to obtain upper bounds on $\sz(n)$.
All of our upper bounds on $\sz(n)$ come from a combination of
splitting and linear programming.
Subsection 5 describes {\it The Thirds Method} for
obtaining large 3-free sets.
It is used to obtain all of our large 3-free sets
beyond where intelligent backtracking could do it.
Subsection 6 describes methods for obtaining large 3-free sets
whose results have been superseded by intelligent backtracking and the Thirds method; nevertheless,
they may be useful at a later time.
They have served as a check on our other methods.

\subsection{The Base 3 Method}\label{se:B3}

Throughout this section $\sz(n)$ will be the largest 3-free
set of $\{0,\ldots,n-1\}$ instead of $\{1,\ldots,n\}$.

The following method appeared in~\cite{ET-1936} but they do not take credit for
it; hence we can call it folklore.
Let $n\in \nat$. Let 
$$A_n= 
\{ m \mid 0\le m\le n 
\hbox{ and all the digits in the base 3 representation of $m$ are in the set $\{0,1\}$ }
\}.
$$
We will later show that $A_n$ is 3-free and 
$|A_n|\approx 2^{\log_3 n}=n^{\log_3 2}\approx n^{0.63}$.

\noindent
{\bf Example:} Let $n=92=1\times 3^4 + 0\times 3^3+ 1\times 3^2 + 0\times 3^1 + 2\times 3^0$. 
Hence $n$ in base 3 is $10102$.
We list the elements of $A_{92}$ in several parts.
\begin{enumerate}
\item
The elements of $A_{92}$ that have a 1 in the fifth place are
$\{10000, 10001, 10010, 10011, 10100, 10101 \}$.
This has the same cardinality as the set
$\{0000, 0001, 0010, 0011, 0100, 0101 \}$ which is $A_{11}$ ($11$ in base 3 is $00102$ --- 92 with the leading 1
changed to 0).
\item
The elements of $A_{92}$ that have a 0 in the fifth place
are the $2^4$ numbers $\{0000, 0001, \ldots, 1111 \}$.
\end{enumerate}

The above example illustrates how to count the size of $A_n$.
If $n$ has $k$ digits in base 3 then there are clearly $2^{k-1}$
elements in $A_n$ that have 0 in the $k$th place.
How many elements of $A_n$ have a 1 in the $k$th place?
In the case above it is $|A_{n-3^{k-1}}|$.
This is not a general formula as the next example shows.

\noindent
{\bf Example:} Let $n=113=1\times 3^4 + 2\times 3^3+ 1\times 3^2 + 1\times 3^1 + 2\times 3^0$. 
Hence $n$ in base 3 is $12112$.
We list the elements of $A_{113}$ in several parts.
\begin{enumerate}
\item
The elements of $A_{113}$ that have a 1 in the fifth place are
$\{10000, 10001, 10010, 10011, 10100, 10101, \ldots, 11111 \}$.
This has $2^5$ elements.
\item
The elements of $A_{113}$ that have a 0 in the fifth place
are the $2^4$ numbers $\{0000, 0001, \ldots, 1111 \}$.
\end{enumerate}

The above example illustrates another way to count the size of $A_n$.
If $n$ has $k$ digits in base 3 then there are clearly $2^{k-1}$
elements in $A_n$ that have 0 in the $k$th place.
How many elements of $A_n$ have a 1 in the $k$th place?
Since $11111\le 12112$ every sequence of 0's and 1's of length 5
is in $A_{113}$.

The two examples demonstate the two cases that can occur in
trying to determine the size of $A_n$.
The following definition and theorem formalize this.

\begin{definition}
Let $S$ be defined as follows.
Let $n\in \nat$. 
Let $k$ be the number of base 3 digits in $n$.
(Note that $k=\floor{\log_3 n}+1$.)
\begin{itemize}
\item
$S(0)=1$ and

\item
$$
S(n)=2^{k-1} + \cases {  2^{k-1}    & if $3^{k-1}+\cdots+3^0 \le n$ ;\cr
                     S(n-3^{k-1}) & otherwise.\cr
}
$$
\end{itemize}
\end{definition}

\begin{theorem}\label{th:B3}~
Let $n\in \nat$ ($n$ can be 0)
\begin{enumerate}
\item
$A_n$ has size $S(n)$.
\item
$A_n$ is 3-free.
\end{enumerate}
\end{theorem}

\begin{proof}

\noindent
1) We show that $A_n$ is of size $S(n)$ by induction on $n$.
If $n=0$ then $A_0 = \{0\}$ which is of size $S(0)=1$.

Inductively assume that, for all $0\le m<n$, $A_m$ is of size $S(m)$.

Let $k$ be the number of base 3 digits in $n$.
There are several cases.
\begin{enumerate}
\item
$n\ge 3^{k-1} + \cdots + 3^0$.
Note that {\it every } element of (in base 3)
$$\{ 0\cdots 0, 0\cdots 1, 0\cdots 10, 0\cdots 11, \ldots, 1\cdots 11 \}$$
(all numbers of length $k$) is in $A_n$, and $A_n$ cannot have any more
elements. Hence $A_n$ is of size $2^k=S(n)$.
\item
$n< 3^{k-1} + \cdots + 3^0$. Note that the $k$th digit in base 3 is 1
since if it was 2 we would be in case 1, and if it was 0 then 
the number would only need $k-1$ (or less) digits in base 3.
\begin{enumerate}
\item
We count the numbers of the form $1b_{k-1}\cdots b_0$ such that
$b_{k-1},\ldots,b_0\in \bit$  and $1b_{k-1},\ldots,b_0\le n$.
This is equivalent to asking that the number (in base 3)
$b_{k-1} \cdots b_0 \le n-3^k$, which is in $A_{n-3^k}$.
Hence we have a bijection between the elements of $A_n$
that begin with 1 and the set $A_{n-3^{k-1}}$.
Inductively this is $S(n-3^k)$. 
\item
We count the numbers of the form $0b_{k-1}\cdots b_0$ such that
$b_{k-1},\ldots,b_0\in \bit$  and $1b_{k-1},\ldots,b_0\le n$.
Since the $k$th digit in base 3 of $n$ is 1, there are clearly
$2^{k-1}$ elements of this form.
\end{enumerate}
Hence we have $A_n$ is of size $S(n-3^{k-1}) + 2^k = S(n)$.
\end{enumerate}

\bigskip

\noindent
2) We show that $A_n$ is 3-free.
Let $x,y,z\in A_n$ form a 3-AP.
Let $x,y,z$ in base 3 be
$x=x_{k-1}\cdots x_0$,
$y=y_{k-1}\cdots y_0$, and
$z=z_{k-1}\cdots a_0$,
By the definition of $A_n$, for all $i$, $x_i, y_i, z_i \in \bit$.
By Fact~\ref{fa:twoy}
$x+z=2y$.
Since $x_i,y_i,z_i \in \bit$ the addition is done {\it without carries}.
Hence we have, for all $i$,
$x_i+z_i=2y_i$. Since $x_i,y_i,z_i \in \bit$
we have $x_i=y_i=z_i$, so $x=y=z$.
\end{proof}

\subsection{Simple Upper Bounds via Splitting}\label{se:good}

\begin{theorem}\label{th:good}~
\begin{enumerate}
\item
For all $n_1,n_2$,
$\sz(n_1+n_2)\le \sz(n_1)+\sz(n_2)$.
\item
For all $n$,
$\sz(kn)\le k\cdot \sz(n)$.
\end{enumerate}
\end{theorem}

\begin{proof}

\smallskip
\noindent
1) Let $A$ be a 3-free subset of $[n_1+n_2]$.
Let $A_1=A\cap [1,n_1]$ and $A_2=A\cap [n_1+1,n_2]$.
Since $A_1$ is a 3-free subset of $[n_1]$, $|A_1|\le \sz(n_1)$.
Since $A_2$ is the translation of a 3-free subset of $[n_2]$, $|A_2|\le \sz(n_2)$.
Hence

$|A|=|A_1|+|A_2|\le \sz(n_1)+\sz(n_2)$.

\smallskip
\noindent
2) This follows from part (1).
\end{proof}

Since we will initially not know $\sz(n_1)$ and $\sz(n_2)$,
how can we use this theorem?
We will often know upper bounds on $\sz(n_1)$ and $\sz(n_2)$
and this will provide upper bounds on $\sz(n_1+n_2)$.

Assume we know upper bounds on $\sz(1),\ldots,\sz(n-1)$.
Call those bounds $\usz(1),\ldots,\usz(n-1)$.
Then $\usz(n)$, defined below, bounds $\sz(n)$.

$\usz(n) = \min\{ \usz(n_1)+\usz(n_2) \mid n_1+n_2 = n \}$

This is the only elementary method we have for getting upper bounds on $\sz(n)$.
We will look at a sophisticated method, which only works for rather large $n$,
in Section~\ref{se:upper}.

\subsection{Exact Values via Intelligent Backtracking}\label{se:back}

In this section we describe several backtracking algorithms for
finding $\sz(n)$.
All of them will be a depth-first search.
The key differences in the algorithms lie in both how much information they
have ahead of time and the way they prune the backtrack tree.
Most of the algorithms find $\sz(1),\ldots,\sz(i-1)$ before
finding $\sz(i)$.

Throughout this section we will
think of elements of $\bits *$ and finite sets of
natural numbers interchangeably.
The following notation makes this rigorous.

\noindent
{\bf Notation:}
Let $\sigma \in \bits n$.
\begin{enumerate}
\item
We identify $\sigma$ with the set $\{ i \mid \sigma(i)=1 \}$.
\item
If $1\le i\le j \le n$ then we denote
$\sigma(i)\cdots \sigma(j)$ by $\sigma[i\ldots j]$.
\item
$\sigma$ {\it has a 3-AP}  means there exists a 3-AP 
$x,y,z$ such that
$\sigma(x)=\sigma(y)=\sigma (z)=1$.
\item
$\sigma$ {\it is 3-free } means that $\sigma$ does not have a 3-AP.
\item
$\#(\sigma)$ is the number of bits set to 1 in $\sigma$. Note that it is the
number of elements in the set we identify with $\sigma$.
\item
Let $\sigma=\alpha\tau$ where $\alpha,\tau \in \bits *$.
Then $\alpha$ is a {\it  prefix} of $\sigma$, and $\tau$ is a {\it suffix} of $\sigma$.
\end{enumerate}

\bigskip

We will need an algorithm to test if a given string is 3-free.
Let $\TFREE$ be such a test. We will describe our
implementation of this in Section~\ref{se:xor}.

For all of the algorithms in this section we will present a main algorithm
that calls a DFS, and then present the DFS.

\subsubsection{Basic Backtracking Algorithms}\label{se:basicback}

In our first algorithm for $\sz(n)$, we do a depth first search
of $\bits n$ where we eliminate a node $\alpha$ if $\alpha$ is not 3-free.

\begin{tabbing}
\qquad\=\qquad\=\qquad\=\qquad\=\qquad\=\qquad\=\qquad\=\+\kill
$\BASIC(n)$\+\\
	$\sz(n)=0$\\
	$\BASICDFS(\epsilon,n)$\\
	Output$(\sz(n))$\-\\
END OF ALGORITHM
\end{tabbing}

\begin{tabbing}
\qquad\=\qquad\=\qquad\=\qquad\=\qquad\=\qquad\=\qquad\=\+\kill
$\BASICDFS(\alpha,n)$\+\\
If $|\alpha|=n$ then\+\\
	$\sz(n)=\max(\sz(n),\#(\alpha))$\-\\
Else\+\\
	$\BASICDFS(\alpha0,n)$ (Since $\alpha$ is 3-free, so is $\alpha0$)\\
	If $\TFREE(\alpha1)$ then $\BASICDFS(\alpha1,n)$\-\-\\
END OF ALGORITHM
\end{tabbing}

The algorithm presented above will find $\sz(n)$ but is inefficient.
The key to the remaining algorithms in this section
is to cut down on the number of nodes visited.
In particular, we will not pursue $\alpha 0$ if we can guarantee
that any 3-free suffix of $\alpha 0$ will not have enough 1's in it
to make it worth pursuing.

Assume we know $\sz(1),\ldots,\sz(n-1)$.
By Theorem~\ref{th:good},
$\sz(n) \in \{\sz(n-1),\sz(n-1)+1\}$.
Hence what we really need to do is see if it is possible for
$\sz(n) = \sz(n-1)+1$.

Assume $A\in \bits n$ is a 3-free set with $\#(A)=sz(n-1)+1$ and prefix $\alpha$.
Then

$$A= \alpha \tau \hbox{ where }|\tau|=n-|\alpha| \hbox{ and }$$

$$\#(\alpha) + \#(\tau) = \sz(n-1)+1.$$

Since $\tau$ is 3-free we know that
$\#(\tau) \le \sz(n-|\alpha|)$.
Therefore if $\alpha$ is the prefix of a 3-free set of $[n]$ of
size $\sz(n-1)+1$ then

$$\#(\alpha) + \sz(n-|\alpha|) \ge \sz(n-1)+1$$


\noindent
{\it Notation:}
$$\POTB(\alpha,n)=\cases  {\TRUE & if $\#(\alpha)+\sz(n-|\alpha|) \ge \sz(n-1)+1$;\cr
                            \FALSE & otherwise.\cr
}
$$
The POT stands for Potential: does $\alpha$ have the potential to be
worth pursuing? The B stands for Basic, since we are using it in the Basic
algorithm.

We now have two tests to eliminate prefixes:
$\TFREE(\alpha)$ and $\POTB(\alpha,n)$.
If $\alpha$ ends in a 0 then we do not need to test $\TFREE(\alpha)$, and
if $\alpha$ ends in a 1 then we do not need to test $\POTB(\alpha,n)$.

\begin{tabbing}
\qquad\=\qquad\=\qquad\=\qquad\=\qquad\=\qquad\=\qquad\=\+\kill
$\BASICtwo(n)$\+\\
	$\sz(n)=\sz(n-1)$\\
	$\BASICDFStwo(\epsilon,n)$\\
	Output$(\sz(n))$\-\\
END OF ALGORITHM
\end{tabbing}

\begin{tabbing}
\qquad\=\qquad\=\qquad\=\qquad\=\qquad\=\qquad\=\qquad\=\+\kill
$\BASICDFStwo(\alpha,n)$\+\\
	If $|\alpha|=n$ then\+\\
		if $\#(\alpha)=\sz(n-1)+1$ then\+\\
			$\sz(n)=\sz(n-1)+1$\\
			Exit $\BASICDFStwo$ and all recursive calls of it\-\-\\
	Else\+\\
		If $\POTB(\alpha0,n)$ then $\BASICDFStwo(\alpha0,n)$\\
		If $\TFREE(\alpha1)$ then $\BASICDFStwo(\alpha1,n)$\-\-\\
END OF ALGORITHM
\end{tabbing}

\subsubsection{Backtracking Algorithm with Information}

\begin{definition}
For all $i\in\nat$
let $\SZ(i)$ be the set of all 3-free sets of $[i]$.
\end{definition}

Let $L$ and $m$ be parameters.
We will later take them to be $L=25$ and $m=80$.
We will do the following
to obtain information in two phases which 
will be used to prune the depth first search tree.

\noindent
{\it Phase I:} Find $\SZ(L)$.

\noindent
{\it Phase II:} For each $\sigma\in \SZ(L)$,
for each $n\le m$, find the size of the largest
3-free set of $\bits {L+n}$ that begins with $\sigma$.

\noindent
{\bf Phase I: Find $\SZ(L)$}

In phase I we find {\it all} 3-free sets of $[L]$ by using
the following recurrence.
We omit the details of the program.

\begin{tabbing}
\qquad\=\qquad\=\qquad\=\qquad\=\qquad\=\qquad\=\qquad\=\+\kill
$\SZ(0)=\{\epsilon \}$\\
$\SZ(L)=\{\alpha0\mid\alpha\in\SZ(L-1) \}\union\{\alpha1\mid\alpha\in\SZ(L-1)\wedge\TFREE(\alpha1)\}$\\
\end{tabbing}

\noindent
{\bf Phase II: Generating More Information }

In this phase we gather
the following information:
for every $\sigma\in \SZ(L)$, for every $n\le m$,
we 
find the $\rho\in \bits n$ such that $\TFREE(\sigma\rho)$ and $\#(\rho)$ is maximized; then
let $\NUM(\sigma,n)=\#(\rho)$.
Note that $\NUM(\sigma,n)$ is the maximum number of 1's that can be in a length-$n$ extension of $\sigma$ while
keeping the entire string 3-free.
The 1's in $\sigma$ do not count.
The main point of the phase is to find
$\NUM(\sigma,n)$ values;
we do not keep the $\rho$'s that are encountered.
We do not even calculate $\sz$ values in the algorithm; however,
we can (and do) easily calculate some $\sz$ values after this phase.

It is easy to see that, for all $\sigma\in \SZ(L)$, $\NUM(\sigma,0)=0$.
Hence we only discuss the case $n\ge 1$.
The algorithm will be given an input $n$, $1\le n \le m$ and
will try to find, for every $\sigma\in \SZ(L)$,
$\NUM(\sigma,n)$.

Before trying to find $\NUM(\sigma,n)$, where $1\le n \le m$,
we have computed the following:
\begin{enumerate}
\item
$\SZ(L)$ from phase I.
\item
For all $\sigma' \in \SZ(L)$, for every $n'<n$, $\NUM(\sigma,n')$.
\end{enumerate}

It is easy to see that $\NUM(\sigma,n) \in \{\NUM(\sigma,n-1),\NUM(\sigma,n-1)+1\}$.
Let $\alpha\in \bits {\le L+m}$ be such that $\sigma$ is a prefix of $\alpha$.
We will want to pursue strings $\alpha$ that have a chance
of showing $\NUM(\sigma,n)=\NUM(\sigma,n-1)+1$.

Assume $A\in \bits {L+n}$ is such that $A$ is 3-free, $A$ has prefix $\alpha$ (hence prefix $\sigma$),
and 
$$\#(A)=\NUM(\sigma,n-1)+1+\#(\sigma).$$
Note that such an $A$ will show that $\NUM(\sigma,n)=\NUM(\sigma,n-1)+1$ with the
last $n$ bits of $A$ playing the role of $\rho$ in the definition of $\NUM(\sigma,n)$.
Rewrite $\alpha$ as $\beta\sigma'$ where $\beta\in\bits *$ and $\sigma'\in \bits L$.
Note that

$$A=\beta\sigma'A[|\beta|+L+1\ldots n+L]=\alpha A[|\beta|+L+1\ldots n+L].$$

Hence

$$
\#(A)=\#(\alpha)+\#(A[|\beta|+L+1\ldots n+L]).
$$

We bound $\#(A)$ from above. Since we know $\alpha$, we know $\#(\alpha).$
(Now is the key innovation.)
Note that $A[|\beta|+L+1\ldots n+L]$ is a string of length $n-|\beta|$
such that $\sigma'A[|\beta|+L+1\ldots n+L]$ is 3-free.
Hence

$$\#(A[|\beta|+L+1\ldots n+L]) \le \NUM(\sigma',n-|\beta|)$$

therefore

$$\#(A)=\#(\alpha) + \#(A[|\beta|+L+1\ldots n]) \le \#(\alpha) + \NUM(\sigma',n-|\beta|).$$

By our assumption, $\#(A)=\NUM(\sigma,n-1)+1+\#(\sigma)$, so 

$$
\NUM(\sigma,n-1)+1+\#(\sigma)=\#(A)\le \#(\alpha)+\NUM(\sigma',n-|\beta|).
$$

Hence

$$\#(\alpha)+ \NUM(\sigma',n-|\beta|) \ge \NUM(\sigma,n-1)+1+\#(\sigma).$$

We define a potential function that uses this test.
$$
\POTG(\sigma,\alpha,n)=\cases{  \TRUE & if $\#(\alpha)+\NUM(\sigma',n-|\beta|) \ge \NUM(\sigma,n-1)+1+\#(\sigma)$;\cr
                               \FALSE & otherwise.\cr
}
$$

\begin{tabbing}
\qquad\=\qquad\=\qquad\=\qquad\=\qquad\=\qquad\=\qquad\=\+\kill
$\INITGATHER(n)$\+\\
	For every $\sigma\in \SZ(L)$\+\\
		$\NUM(\sigma,0)=0$-\-\-\\
END OF ALGORITHM\\
\end{tabbing}

\begin{tabbing}
\qquad\=\qquad\=\qquad\=\qquad\=\qquad\=\qquad\=\qquad\=\+\kill
$\GATHER(n)$ (Assume $n\ge 1$.)\+\\
	For every $\sigma\in \SZ(L)$\+\\
		$\NUM(\sigma,n)=\NUM(\sigma,n-1)$\-\\
	$\GATHERDFS(\sigma,\epsilon,n)$\-\\
END OF ALGORITHM
\end{tabbing}

\begin{tabbing}
\qquad\=\qquad\=\qquad\=\qquad\=\qquad\=\qquad\=\qquad\=\+\kill
$\GATHERDFS(\sigma,\alpha,n)$ ($\sigma$ is of length $L$)\+\\
	If $|\alpha|=n$ then\+\\
		If $\#(\alpha)=\NUM(\sigma,n-1)+1$ then \+\\
			$\NUM(\sigma,n)=\NUM(\sigma,n-1)+1$\\
			Exit $\GATHERDFS$ and all recursive calls of it.\-\-\\
	Else\+\\
		If $\POTG(\sigma,\alpha0,n)$ then $\GATHERDFS(\sigma,\alpha0,n)$\-\-\\
END OF ALGORITHM
\end{tabbing}

Now that we have the values $\NUM(\sigma,n)$ for all $n$, $0\le n\le m$
we can compute $\sz(n)$.  

\begin{tabbing}
\qquad\=\qquad\=\qquad\=\qquad\=\qquad\=\qquad\=\qquad\=\+\kill
$\FINDsz(n)$\+\\
If $n\le L$ then $\sz(n)=\max\{ \#(\sigma[1..n]) \mid \sigma \in \SZ(L) \}$\\
If $L<n\le m+L$ then $\sz(n) = \max\{ \NUM(\sigma,n-L) \mid \sigma \in \SZ(L) \}$\-\\
END OF ALGORITHM
\end{tabbing}

\noindent
{\bf Phase III: Using the Information Gathered}

We will present the algorithm for $n>m$.
We devise a potential function for prefixes.

Assume $A\in \bits n$ is a 3-free set with $\#(A)=\sz(n-1)+1$ and prefix $\alpha$.
Rewrite $\alpha$ as $\beta\sigma'$ where $\beta\in\bits *$ and $\sigma'\in \bits L$.
Note that

$$A=\beta\sigma' A[|\beta|+L+1\ldots n].$$

Hence

$$
\#(A) = \#(\beta\sigma')+\#(A[|\beta|+L+1\ldots n])= \#(\alpha) + \#(A[|\beta|+L+1\ldots n]).
$$

We bound $\#(A)$ from above.
Clearly we know $\#(\beta\sigma')=\#(\alpha).$
(Now is the key innovation.)
Note that $\sigma' A[|\beta|+L+1\ldots n]$ is a 3-free string of length $L+n-|\beta|-L = n-|\beta|$
that has $\sigma'\in \SZ(L)$ as a prefix. Hence

$$\#(A[|\beta|+L+1\ldots n]) \le \NUM(\sigma',n-|\beta| - L).$$

Therefore

$$\#(A) \le \#(\alpha) + \NUM(\sigma',n-|\beta|-L).$$

Since $\#(A)=\sz(n-1)+1$, we have

$$\sz(n-1)+1 = \#(A)  \le \#(\alpha) + \NUM(\sigma',n-|\beta|-L).$$

Hence

$$\#(\alpha) + \NUM(\sigma',n-|\beta|-L) \ge \sz(n-1)+1.$$

If $n-|\beta| - L\le m$ then $\NUM(\sigma',n-|\beta|-L)$ has been computed and  we use this test.
If $n-|\beta| - L >m$ then we cannot use this test; however in this case there are several 
weaker bounds we can use.

\noindent
{\bf Test $T1$:} We use $\sz$. Since

$$\#(A[|\beta|+L+1\ldots n]) \le \sz(n-|\beta| - L)$$

we define $T1(\alpha)$ as follows

$$T1(\alpha): \#(\alpha) + \sz(n-|\beta|-L) \ge \sz(n-1)+1.$$

Note that this is the same test used in $\POTB$.

\noindent
{\bf Test $T2$:} We use $\NUM$ and $\sz$. Note that $A[|\beta|+1,\ldots,|\beta|+m]$
is a string of length $m$ that extends $\sigma'$. Hence

$$\#(A[|\beta|+1,\ldots,|\beta|+m]) \le \NUM(\sigma',m).$$

Clearly

$$\#(A[|\beta|+m+1,\ldots,n]) \le \sz(n-m-|\beta|).$$

Hence

$$\#A(|\beta|+1,\ldots,n) \le \NUM(\sigma',m) + \sz(n-m-|\beta|).$$

We define $T2(\alpha)$ as follows:

$$T2(\alpha): \#(\alpha) + \NUM(\sigma',m) + \sz(n-m-|\beta|)\ge \sz(n-1)+1.$$

\noindent
{\bf Test $T3$:} We use forbidden numbers. In Section~\ref{se:xor} we will see that
associated with $\alpha$ will be forbidden numbers. These are all the $f$,  
$|\alpha|<f\le n$
such that, viewing $\alpha$ as a set,  $\alpha \union \{f\}$ has a 3-AP.
Let $c$ be the number of numbers that are {\it not} forbidden.
If $\alpha$ can be extended to a 3-free set of $[n]$ that has
$\sz(n-1)+1$ elements in it then we need the following to be true.

$$T3(\alpha): \#(\alpha) + c \ge \sz(n-1)+1.$$

\vfill\eject

\noindent
{\it Notation:}
Let $\sigma'\in \bits L$,
$\alpha\in \bits *$, $n\in \nat$, and $|\alpha|<n$.
Let $\alpha=\beta\sigma'$.
Then
$$
\POTF(\alpha,n)=\cases {\TRUE & if $n-|\beta|-L \le m$ and \cr
                              & \qquad $\#(\alpha) + \NUM(\sigma',n-|\beta|-L) \ge \sz(n-1)+1;$\cr
                        \TRUE & if $n-|\beta|-L > m$ and \cr
                              & \qquad $T1(\alpha) \wedge T2(\alpha)\wedge T3(\alpha)$;\cr
                        \FALSE & otherwise.\cr
}
$$

\begin{tabbing}
\qquad\=\qquad\=\qquad\=\qquad\=\qquad\=\qquad\=\qquad\=\+\kill
$\FINAL(n)$\+\\
	$\sz(n)=\sz(n-1)$\\
	For every $\sigma\in \SZ(L)$\+\\
		$\FINALDFS(\sigma,n)$\-\\
	Output$(\sz(n))$\-\\
END OF ALGORITHM
\end{tabbing}

\begin{tabbing}
\qquad\=\qquad\=\qquad\=\qquad\=\qquad\=\qquad\=\qquad\=\+\kill
$\FINALDFS(\alpha,n)$\+\\
	If $|\alpha|=n$ then\+\\
		If $\#(\alpha)=\sz(n-1)+1$ then\+\\
			$\sz(n)=\sz(n-1)+1$\\
			Exit $\FINALDFS$ and all recursive calls to it\-\-\\
	Else (In what follows we know $|\alpha|<n$.)\+\\
		If $\POTF(\alpha0,n)$ then $\FINALDFS(\alpha0,n)$\\
		If $\TFREE(\alpha1)$ then $\FINALDFS(\alpha1,n)$\-\-\\
END OF ALGORITHM
\end{tabbing}

\subsubsection{Testing if a string is 3-free}\label{se:xor}

In the above algorithms we called a procedure called $\TFREE$.
We do not have such a procedure. Instead we have a process
that does the following. 

\begin{itemize}
\item
A string is being constructed bit by bit.
\item
While constructing it we need to know if adding a 1
will cause it to no longer be 3-free.
\end{itemize}

We describe this process.

\begin{enumerate}
\item
We are building $\alpha$ which will be a string of length at most $n$.
We maintain both the string $\alpha$
and the array of forbidden bits $\forb$.
\item
Assume $\alpha$ is currently of length $i$.
If $k\ge i+1$ and $\forbs k =1$ then
setting $\alpha(k)=1$ would create a 3-AP in $\alpha$.
\item
Initially $\alpha$ is of length 0 and $\forb$ is an array of $n$ 0's.
\item
(This is another key innovation.)
Assume that we have set $\alpha(1)\cdots\alpha(i-1)$.
Conceptually maintain $\alpha$ and $\forb$ as follows

\[
\begin{array}{ccccccccc}
          &         &                &\alphas 1    &\alphas 2    &\alphas 3     & \cdots &\alphas{i-2} & \alphas{i-1}  \cr
\forbs n  & \cdots  &  \forbs {2i+1} & \forbs {2i} &\forbs {2i-1}& \forbs{2i-2} & \cdots &\forbs{i+1}  & \forbs{i}\cr
\end{array}
\]

\item
If we append 0 to $\alpha$ then the new $\alpha$ and $\forb$ are

\[
\begin{array}{cccccccccc}
         &         &              &\alphas 1     &\alphas 2    &\alphas  3      & \cdots &\alphas {i-2} & \alphas{i-1}& 0  \cr
\forbs{n} & \cdots  &  \forbs{2i+2} & \forbs{2i+1} &\forbs{2i-2}& \forbs{2i-1} & \cdots &\forbs{i+3} &\forbs{i+2}  & \forbs{i+1}\cr
\end{array}
\]

\item
If we want to append 1 to $\alpha$ we do the following:

\begin{enumerate}
\item
Shift $\forb$ one bit to the right.

\[
\begin{array}{ccccccccc}
         &         &              &\alphas 1     &\alphas 2  &\alphas 3    & \cdots &\alphas {i-2} & \alphas{i-1}         \cr
\forbs{n} & \cdots  &  \forbs{2i} & \forbs{2i-1} &\forbs{2i-2}& \forbs{2i-3} & \cdots &\forbs{i+2}  &\forbs{i+1}           \cr
\end{array}
\]

\item
The bit string $\alpha$ remains as the above diagram, and 
$\forb$ is replaced by the bitwise OR of $\alpha$ and $\forb$.
(The bits of $\forb$ that do not correspond to bits of $\alpha$ remain the same.)
We denote the new $\forb$ by $\forbp$.

\item
Shift $\alpha$ one bit to the left and append a 1 to it.

\[
\begin{array}{ccccccccc}
 &         &\alphas 1      &\alphas 2        &\alphas 3         & \cdots           &\alphas{i-2}    & \alphas{i-1}   & 1     \cr
 \forbps{n} & \cdots  &  \forbps{2i}   & \forbps{2i-1} &\forbps{2i-2}    & \cdots & \forbps{i+3} & \forbps{i+2}  &\forbps{i+1}           \cr
\end{array}
\]

\end{enumerate}
\end{enumerate}

We leave it to the reader to verify that this procedure correctly sets $\forb$.
Note that this procedure is very fast since the main operations are  bit-wise ORs
and SHIFTs.

In the DFS algorithms above we often have the line

If $\TFREE(\alpha1)$ then $\DFS(\alpha1)$
(where $\DFS$ is one of the $\DFS$ algorithms).

As noted above we do not have a procedure $\TFREE$.
So what do we really do?
We use the forbidden bit array.
For example, lets say that the first 99 bits of $\alpha$ are known
and the forbidden bit pattern from 100 to 108 is as follows.

\[
\begin{array}{ccccccccccccc}
 &             & \cdots  &100   &101   &102   & 103  &104  &105 & 106 & 107 & 108 \cr
 &\forbps{n}   & \cdots  &1     &1     &1     &1     &0    &0   & 1   & 0   & 0   \cr
\end{array}
\]

We are pondering extending $\alpha$ by 0 or 1.
But note that the next place to extend $\alpha$ is a forbidden bit.
In fact, the next 4 places are all forbidden bits.
Hence we automatically put 0's in the next four places.
After that we do the recursive calls to the DFS procedure.

We illustrate this by showing how we really would code
$\BASICDFS$.

\begin{definition}
Let $\alpha,\forb \in \bits *$ 
be such that $\forb$ is the forbidden bit array for $\alpha$.
Let $b\in \bit$.
Then $\ADJUST(\alpha,\forb,b)$
is the forbidden bit array that is created when $b$ is appended to $\alpha$.
The details were described above.
\end{definition}

\begin{tabbing}
\qquad\=\qquad\=\qquad\=\qquad\=\qquad\=\qquad\=\qquad\=\+\kill
$\BASICDFS(\alpha,\forb,n)$\+\\
	If $|\alpha|=n$ then\+\\
		$\sz(n) = \max\{\sz(n),\#(\alpha)\}$
		Exit $\BASICDFS$\-\\
	Else\+\\
		While ($\forb_{|\alpha|+1}=1$) and ($|\alpha|\le n$)\+\\
			$\alpha = \alpha0$\-\\
		$\BASICDFS(\alpha0,\ADJUST(\alpha,\forb,0),n)$
		$\BASICDFS(\alpha1,\ADJUST(\alpha,\forb,1),n)$\-\-\\
END OF ALGORITHM
\end{tabbing}

\subsubsection{How we Really Coded this up}

If there is a 3-free set $A\in\bits n$ such that $\#(A)=\sz(n-1)+1$ then
$A(1)=A(n)=1$ (otherwise there would be a 3-free subset of $[n-1]$
of size $\sz(n-1)+1$).
We use this as follows.
\begin{enumerate}
\item
In $\BASIC$ and $\BASICtwo$ we can start with 1 instead of $\epsilon$.
We can also end with a 1.
\item
In $\FINAL$ we need only begin with the $\sigma\in \SZ(L)$ that
begin with 1.
($\GATHER$ is unaffected since we need to gather information about
all $\sigma$ including those that begin with 0.)
\item
In the procedure $\TFREE$ we test if $\sigma$ is 3-free,
we are actually testing if $\sigma \cup \{n\}$ is 3-free.
\end{enumerate}

In the algorithms above we keep trying to improve
$\sz(n)$ even if we have the value $\sz(n-1)+1$.
When coding it up we exited the program when
the value $\sz(n-1)+1$ was obtained.

\subsubsection{Empirical Results}

The test

$$
\#(\alpha)+\NUM(\sigma',n-|\beta|-L) \ge \sz(n-1)+1.
$$

cut down on the number of nodes searched by a factor of 10.
The tests $T1$ and $T2$ were useful but not dramatic.
The test $T3$ did not seem to help much at all.

The method enabled us to find exact values up to
$\sz(\lastnum)$.

\subsection{Upper Bounds via Linear Programming}\label{se:lp}

We rephrase the problem of finding a large 3-free set of $[n]$ 
as an integer programming problem:

\smallskip
\noindent
{\bf Maximize:} $x_1+\cdots+ x_n$

\noindent
{\bf Constraints:}

$x_i + x_j + x_k \le 2$  for $1\le i<j<k \le n$ where $i,j,k$ is a 3-AP.

$0\le x_i\le 1$

Say that $(x_1,\ldots,x_k)$ is a solution.
Then the set
$$A=\{ i \mid x_i=1\}$$
is a 3-free set of size $\sz(n)$.
Hence we can talk about solutions to this integer programming
problem, and 3-free sets $A$, interchangeably.

The general integer programming problem is NP-complete.
We have tried to use IP packages to solve this but the problem
is too big for them. The two we used are actually parts of LP packages,
CPLEX and GLPK.
However, we can use linear programming, and these packages,
to get upper bounds on $\sz(n)$.

If the integer program above is relaxed
to be a linear program, and the max value
for $x_1+\cdots+x_n$ was $s$, then we would know
$\sz(n)\le s$.
We will use this linear program, with many additional constraints, to obtain
upper bounds on values of $\sz(n)$ for which we do not have
exact values.

If we just use the relaxation of the integer programming problem
given in the last section then the upper bounds obtained are worse than
those in the above table.
Hence we will need to add more upper bound constraints.
For example, if we know that $sz(100)\le 27$ and
we are looking at $\sz(200)$ we can put in the constraints

$x_1 + \cdots + x_{100} \le 27$

$x_2 + \cdots + x_{101} \le 27$

$\vdots$

$x_{100} + \cdots + x_{199} \le 27$

$x_{101} + \cdots + x_{200} \le 27$

$x_1 + x_3 + x_5 + \cdots + x_{199} \le 27$

More generally, if we know $\sz(i)$ for $i\le m$ then,
for every $3\le i\le m$, we have the constraints

$$
x_{b_1} + \cdots + x_{b_i} \le \sz(i) \hbox{ such that } b_1<\cdots<b_i \hbox{ is an $i$-AP }.
$$

Putting in all of these constraints caused us linear programs
that took too long to solve. However, the constraints based
on $\sz(100) \le 27$ are intuitively more powerful than the
constraints based on $\sz(3)=2$. 
Hence we put in fewer constraints. However, it turned out that
putting in all constraints that used the values of $\sz(i)$ for
$20\le i \le 186$ yielded programs that ran quickly.
But there was another problem --- these programs always resulted
in numbers bigger than our upper bounds on $\sz(n)$ based
on splitting, hence the information was not useful.

We then put in {\it lower bound constraints}.
For example, if we want to see if $\sz(187)=41$ we can
have the constraint

$$x_1+ \cdots + x_{187} = 41.$$

We can also have constraints based on known lower values
of $\sz$. For example, since $\sz(100)=27$ a 3-free set
of $[187]$ of size 41 would need to have 

$$x_{101} + \cdots + x_{187} \ge 14$$

since otherwise

$$x_1 + \cdots + x_{187} \le 40.$$

We then put in all lower bound constraints.
This always resulted in either
finding the conjectured value
(which was not helpful) or finding
that the feasible region was empty.
In the latter case we know that the conjectured
value cannot occur.

We now formalize all of this.

\noindent
{\it INPUT:}
\begin{itemize}
\item
$n$
\item
$\usz(1),\ldots,\usz(n-1)$ (upper bound on $\sz$).
\item
$t$. (We want to show $\sz(n)< t$.)
\end{itemize}

\noindent
{\it OUTPUT:}
Either ``$\sz(n)\le t-1$'' or ``NO INFO''

We will add the following constraints.

\noindent
{\it New Upper Constraints using Known Values of $\sz$}

For every $i$, $3\le i\le m$, we have the constraints

$$x_{b_1} + \cdots + x_{b_i} \le \sz(i) \hbox{ such that } b_1<\cdots<b_i \hbox{ is an $i$-AP }.$$

\noindent
{\it New Lower Constraints Based on $\usz(i)$}

{}From the upper bound constraints we have

$$x_{b_1} + \cdots + x_{b_i} \le \sz(i) \hbox{ such that } b_1<\cdots<b_i \hbox{ is an $i$-AP }.$$

If $A$ is to have $t$ elements in it we need

$$\sum_{j\notin \{b_1,\ldots,b_i\} } x_j \ge t-\sz(i) \hbox{ such that } b_1<\cdots<b_i \hbox{ is an $i$-AP }.$$

\noindent
{\it New Lower Constraints Based on Prefixes}

We want to know if there is a 3-free set $A\subseteq \{1,\ldots,n\}$
with $\#(A)\ge t$.
Let $L$ be a parameter.
We consider every $\sigma \in \bits L$ that could be a prefix of $A$.
In order to be a prefix it must satisfy the following criteria
(and even then it might not be a prefix).
\begin{itemize}
\item
$\sigma$ is 3-free.
\item
For every $i$, $1\le i\le L$, let $\tau_i$ be the $i$-length prefix
of $\sigma$. Then
$$\#(\tau_i) + \sz(n-i) \ge t.$$
\item
$\sigma$ begins with a 1.
We can assume this since if there is such a 3-free set that does not not begin with 1 then we can shift it.
\end{itemize}

\begin{definition}
If $\sigma$ satisfies the criteria above then
$\GOOD(\sigma)$ is TRUE, else it is false.
\end{definition}

For each such $\sigma$ such that $\GOOD(\sigma)=\TRUE$
we create a linear program that has
the following additional constraints.

\noindent
$x_i=\sigma(i)$ for all $i$, $1\le i\le L.$

\noindent
$x_{L+1} + x_{L+2} + \cdots + x_n \ge t - \#(\sigma).$

If {\it every} such linear program returns a value that is $\le t-1$
then we can say that $\sz(n)\le t-1$.
If {\it any } return a value that is $\ge t$ then we cannot make any conclusions.

Using $L=30$ we improved many of the upper bounds
obtained by the splitting method.
This value of $L$ was chosen partially because of issues with
word-size.

\subsection{Lower Bounds via Thirds Method}\label{se:thirds}

The large 3-free sets that are found by the methods above
all seem to have many elements in the first and last thirds
but very few in the middle third. This leads to the following
algorithm to find a large 3-free set.
Assume $n$ is divisible by 3.

Given a large 3-free set $A\subseteq [m]$
one can create a 3-free set of $[3m-1]$ in the following way:
$A \union B$ where $B$ is the set $A$ shifted to be a subset
of $\{2m+1,\ldots,3m\}$.
You can then try to include some elements from the middle;
however, most of the elements of the middle will be excluded.

We could take different 3-free sets of $[m]$
for $A$ and $B$.  In fact, we could go through {\it all } large 3-free
sets of $[m]$.

In practice we do not use the maximum three free set of $[m/3]$. 
We sometimes found larger 3-free sets of $[m]$ by using 3-free sets of size
between $m/3 - \log m$ and $m/3 + \log m$ that are of size within one or two of maximum.
This leads to most of the remaining
middle elements being forbidden; hence, searching for the optimal
number that can be placed is easy.
There is nothing sacrosanct about $\log m$ and being within one or two of maximum.
We only used this technique for numbers between 3 and 250; for
larger values of $m$ other paramters may lead to larger 3-free sets.
We do note that for $m\le \lastnum$ --- values for which we know $\sz(m)$
exactly --- the thirds method always found a set of size $\sz(m)$.

\begin{enumerate}
\item
$\sz(194)\ge 41$. 
(This was known by \cite{Wroblewski-Unknown}a)
\item
$\sz(204)\ge 42$. (This is new.)
\item
$\sz(209)\ge 43$. 
(This was known by \cite{Wroblewski-Unknown}a)
\item
$\sz(215)\ge 44$.  
(This was known by \cite{Wroblewski-Unknown}a)
\item
$\sz(227)\ge 45$.  (This is new.)
\item
$\sz(233)\ge 46$.  (This is new.)
\item
$\sz(239)\ge 47$.  
(This was known by \cite{Wroblewski-Unknown}a)
\item
$\sz(247)\ge 48$.  
(This was known by \cite{Wroblewski-Unknown}a)
\end{enumerate}

\noindent
The three free set that showed $\sz(204)\ge 42$ is  

\noindent
$\{1, 3, 8, 9, 11, 16, 20, 22, 25, 26, 38, 40, 45, 46, 48, 53, 57, 59, 62, 63,$ \\
$127, 132, 134, 135, 139, 140, 147, 149, 150, 152, 156, 179,$\\
$181, 182, 186, 187, 189, 194, 198, 200, 203, 204 \}$.

\noindent
The three free set that showed $\sz(227)\ge 45$ is

\noindent
$\{1, 2, 6, 8, 12, 17, 19, 20, 24, 25, 27, 43, 45, 51, 54, 55, 58, 60, 64, 72, 76, 79,$\\
$129, 145, 147, 154, 155, 159, 160, 167, 169, 170, 172, 176,$\\
$201, 202, 206, 208, 212, 217, 219, 220, 224, 225, 227\}$.

\noindent
The three free set that showed $\sz(233)\ge 46$ is

\noindent
$\{1, 4, 5, 11, 13, 14, 16, 26, 29, 30, 35, 50, 52, 58, 61, 62, 68, 73, 76, 77, 80, 82, 97,$\\
$137, 152, 154, 157, 158, 161, 166, 172, 173, 176,$\\
$182, 184, 199, 204, 205, 208, 218, 220, 221, 223, 229, 230, 233\}$.

\subsection{Other methods}

We present methods for constructing large 3-free sets that were tried but ended up not
being as good as Intelligent Backtracking or the Thirds Method.
These methods, or modifications of them, may prove useful later.
In addition they were a check on our data.

\subsubsection{The Concatenation Method}\label{se:concat}

The following theorem is similar in proof to Theorem~\ref{th:good}.

\begin{definition}
If $B$ is a set and $m\in \nat$
then an {$m$-translate of $B$} is the set
$\{ x+ m \mid x\in B\}$.
\end{definition}

We need the following simple fact.

\begin{fact}\label{fa:concat}
Let $n=n_1+n_2$.
Let ${\cal A}_1$ be the set of all 3-free subsets of $[n_1]$.
Let ${\cal A}_2$ be the set of all 3-free subsets of $[n_2]$.
If $A$ is a 3-free subset of $[n_1+n_2]$ then $A = A_1 \union A_2$
where $A_1\in {\cal A}_1$ and
$A_2$  is an $n_1$-translate of some element of ${\cal A}_2$.
\end{fact}

\begin{definition}
If $n,k\in\nat$ then
${\cal E}_{n,k}$ is the set of 3-free subsets of $[n]$
that contain both 1 and $n$ and have size $k$.
\end{definition}

The following assertions,
stated without proof, establish the usefulness of the ${\cal E}$'s in
computing $\sz(n)$:
\begin{enumerate}
\item [(a)] $|{\cal E}_{1,0}|=0$, $|{\cal E}_{1,1}|=1$. (This is used at the base of a recursion.)
\item [(b)] if $n\ge 2$ then $|{\cal E}_{n,0}|=0$, $|{\cal E}_{n,1}|=0$, and $|{\cal E}_{n,2}|=1$. (This is used at the base of a recursion.)
\item [(c)] if ${\cal E}_{n,k} \ne \emptyset$ then $\sz(n) \ge k$;
\item [(d)] if ${\cal E}_{n,k} = \emptyset$ where $k,n > 1$ then ${\cal E}_{n,l} = \emptyset$ for all $l > k$; and
\item [(e)] if ${\cal E}_{n,k} = \emptyset$ and $k,n>1$ then $\sz(n) < k$.
\end{enumerate}

The sets that comprise ${\cal E}_{n, k}$ can be obtained from
${\cal E}_{m, l}$ where $m < n$ and $l < k$.  Let
$A \in {\cal E}_{n,k}$.  Partition $A$ into
$A_1 = A \cap \{1, \ldots, \lceil {n \over 2} \rceil\}$ and
$A_2 = A \cap \{\lceil {n \over 2} \rceil + 1, \ldots, n\}$.
Let $x$ be the largest element of $A_1$ and let $y$ be the smallest
element of $A_2$.  Then $A_1 \in {\cal E}_{x, \mid A_1 \mid}$ and
$A_2$ is a $(y-1)$-translation of an element of
${\cal E}_{n - y + 1, \mid A_2 \mid}$.
This can be used to obtain a Dynamic Program to find
${\cal E}_{n,k}$.

This method requires too much time and space to be useful for
finding $\sz(n)$.  However, it is useful if you want to find
many large 3-free sets of $[n]$.

\subsubsection{The Greedy Vertex Cover Method}

We can rephrase our problem as that of finding
the maximum independent set in a hypergraph.

\begin{definition}~
\begin{enumerate}
\item
A {\it hypergraph} is a pair $(V,E)$ such that
$E$ is a collection of subsets of $V$.
The elements of $V$ are called {\it vertices}.
The elements of $E$ are called {\it hyperedges}.
\item
A {\it 3-uniform hypergraph} is one where all of the hyperedges
have exactly three vertices in them.
\item
If $H=(V,E)$ is a hypergraph then $\overline H$, the complement of $H$, is
$(V,{\cal P}(V)-E)$ where ${\cal P}(V)$ is the powerset of $V$.
\item
If $H=(V,E)$ is a hypergraph then an {\it independent set of $H$}
is a set $U\subseteq V$ such that
$$(\forall U'\subseteq U)[U'\notin E].$$
\item
If $H=(V,E)$ is a hypergraph then a {\it vertex cover of $H$}
is a set $U\subseteq V$ such that
$$(\forall e\in E)(\exists v\in U)[v \in E].$$
\end{enumerate}
\end{definition}

\begin{note}
If $U$ is a vertex cover of $H$ then $\overline U$ is an independent
set of $H$.
\end{note}

Let $G=(V,E)$ be the following 3-uniform hypergraph.

\[
\begin{array}{rl}
V=&\{1,\ldots,n\};\cr
E=&\{(i,j,k) \st (i<j<k) \wedge \hbox{$i,j,k$ form a 3-AP} \}.\cr
\end{array}
\]

The largest independent set in this hypergraph corresponds to the
largest 3-free set of $[n]$.
Unfortunately the independent set problem, even for the simple case of graphs,
is NP-complete.
In fact, approximating the maximum independent set is known to be 
NP-hard~\cite{Hastad-1999}.
It is possible that our particular instance is easier.

We have used the greedy method for vertex cover on our hypergraph; the complement of the cover
gives a (not necessarily good) solution quickly.
To compute the greedy vertex cover, at each step one
selects the vertex in $G$ with highest degree (ties are broken randomly).
Once a vertex is selected it
is removed from the graph along with all its incident edges.  This process
continues until no edges remain in  $G$.  For each of the $O(n)$ removals we
find the vertex with highest degree in $O(n)$ time, so the greedy vertex
cover can be found in $O(n^2)$ time.

This method does not give us optimal 3-free sets and hence is not
useful for computing $\sz(n)$.

However, it does give large 3-free sets that are close to optimal --- within
one or two --- and it is fast.

\subsubsection{The Randomization Method}

We describe two methods to produce large 3-free sets that use randomization.

The first method, given below, uses a randomly chosen permutation of
$1, \ldots, n$.

\begin{enumerate}
\item [1)] Randomly permute $1, \ldots, n$ to get $a_1, \ldots, a_n$.
\item [2)] Set $S = \emptyset$.
\item [3)] For $i = 1$ to $n$ add $a_i$ to $S$ if doing so does not create a 3-AP in $S$
\end{enumerate}

BILL- why was space a problem again?

Running time is $O(n^2)$ using appropriate data structures, but the space
requirements are large for large $n$.
(Note that storing a permutation of size $n$ takes $\Theta(n\log n)$ space
which is big for large values of $n$.)
When space is a factor, we use a different method
which keeps track of the 3-free set $S$ and the set of numbers that can
be added to it without introducing a 3-AP.
\begin{enumerate}
\item [1)] Set $S = \emptyset$ and $P = \{1, \ldots, n\}$
\item [2)] While $P \ne \emptyset$
\begin{enumerate}
  \item [a)] randomly select an element $x$ from $P$
  \item [b)] set $S = S \cup \{x\}$
  \item [c)] remove from $P$ all elements that form a 3-AP with $x$ and
             another element of $S$
\end{enumerate}
\end{enumerate}

The second method also runs in time $O(n^2)$ but is empirically slower than
the first method.  However, with the use of an appropriate data structure,
$P$ requires far less storage space than the permutation required by the first
method.

These methods, just like the Greedy Vertex Cover method,
do not give us optimal 3-free sets and hence is not
useful for computing $\sz(n)$.

However, they do give large 3-free sets that are close to optimal,
and they are fast.

\subsection{The Values of $\sz(n)$ for Small $n$}

We have several tables of results for small $n$ in the appendix.
A lower bound of $X$ on $\sz(n)$
means that there is
a 3-free set of $[n]$ of size $X$.
An upper bound of $X$ on $\sz(n)$
means
that no set of $[n]$ of size $X$ is 3-free.

There are three tables of information about $\sz(n)$ for
small $n$ in Appendix II.

\begin{enumerate}
\item
Tables 1 and 2
gives exact values for $\sz(n)$ for
$1\le n\le \lastnum$.
We obtained these results by intelligent backtracking.
\item
Table 3 gives
upper and lower bounds for $\lastnump\le n\le \lastnumtwo$.
The upper bounds for $\lastnump\le n\le \lastnumtwo$ were obtained by Theorem~\ref{th:good}
and the linear programming upper bound technique described in Section~\ref{se:lp}.
The lower bounds for $\lastnump\le n\le \lastnumtwo$ were obtained by the thirds-method
described in Section~\ref{se:thirds}.
\end{enumerate}

\section{What Happens for Large $n$?}\label{se:large}

In this section we look at several methods to construct 3-free sets.
In a later section we will compare these methods to each other.
We present the literature in order of how large the sets produced are,
which is not the order they appeared in historically.

\subsection{3-Free Subsets of Size $n^{0.63}$: The Base 3 Method}\label{se:B3a}

We restate Theorem~\ref{th:B3} here for completeness.

\begin{theorem}\label{th:B3a}
For all $n$, $\sz(n)\ge n^{\log_3 2}\sim n^{0.63}$.
\end{theorem}

\subsection{3-Free Subsets of Size ${n^{0.68-\epsilon}}$: The Base 5 Method}\label{se:B5}

According to~\cite{ET-1936},
G. Szekeres conjectured that $\sz(n) = \Theta(n^{\log_3 2} )$.
This was disproven by Salem and Spencer~\cite{SS-1942} (see below);
however, in 1999 Ruzsa (Section 13 of~\cite{Ruzsa-1999}) noticed
that a minor modification to the proof of the Theorem~\ref{th:B3a}
yields the following theorem which also disproves the conjecture.
His point was that this is
an easy variant of Theorem~\ref{th:B3} so it is surprising that
it was not noticed earlier.

\begin{theorem}
For every $\epsilon>0$ there exists $n_0$ such that, for all $n\ge n_0$,
$\sz(n) \ge n^{(\log_5 3)-\epsilon}\sim n^{0.68 - \epsilon}$.

\end{theorem}

\begin{sketch}
Let $L$ be a parameter to be chosen later.
Let $k=\floor{\log_5 n}-1$.
Let $A$ be the set of positive integers that, when expressed in base 5,
\begin{enumerate}
\item
use at most $k$ digits,
\item
use only 0's, 1's, and 2's, and
\item
use {\it exactly} $L$ 1's.
\end{enumerate}

One can show, using Fact~\ref{fa:twoy},  that $A\subseteq [n]$ and $A$ is 3-free.
If we take $L=\floor{k/3}$ one can show that $|A|\ge n^{(\log_5 3)-\epsilon}$.
\end{sketch}

Consider the following variant of the Base 5 method.
Use Base 5, but use digits $\{-1,0,1\}$ and require that every
numbers has exactly $L$ 0's.  If $(b_{k-1},\ldots,b_0)$ is a number
expressed in Base 5 with digits $\{-1,0,1\}$ and with exactly
$L$ digits 0, then $\sum_{i=0}^{k-1} b_i^2 = n-L$.
This method, expressed this way, is a our version of the Sphere Method
(see Section~\ref{se:sphere})
with parameters $d=1$ and $s=n-L$.

\subsection{3-Free Subsets of Size $n^{1-\frac{1+\epsilon}{\lg\lg n}}$: The KD Method}\label{se:KD}

The first disproof of Szekeres's conjecture (that $\sz(n) = \Theta(n^{\log_3 2} )$)
was due to Salem and Spencer~\cite{SS-1942}.

\begin{theorem}
For every $\epsilon>0$ there exists $n_0$ such that, for all $n\ge n_0$,
$\sz(n) \ge n^{1-\frac{1+\epsilon}{\lg\lg n}}$
\end{theorem}

\begin{definition}
Let $d,n\in \nat$. Let $k=\floor{\log_{2d-1} n}-1$.
Assume that $d$ divides $k$.
$\KD_{d,n}$ is the set of all $x\le n$ such that
\begin{enumerate}
\item
when expressed in base $2d-1$ only uses the digits $0,\ldots,d-1$, and
\item
each digit appears the same number of times, namely $k/d$.
\end{enumerate}
\end{definition}

We omit the proof of the following lemma.

\begin{lemma}\label{le:KD}
For all $d,n$ $\KD_{d,n}$ is 3-free.
\end{lemma}

\begin{theorem}\label{th:ss}
For every $\epsilon>0$ there exists $n_0$ such that,
for all $n\ge n_0$,
$\sz(n) \ge n^{1-\frac{1+\epsilon}{\lg\lg n}}$.
\end{theorem}

\begin{sketch}
An easy calculation shows that, for any $d,n$,
$\KD_{d,n}\subseteq [n]$.
By Lemma~\ref{le:KD} $\KD_{d,n}$ is 3-free.
Clearly

$$|\KD_{d,n}| = \frac{k!}{[(k/d)!]^d}.$$

By picking $d$ such that $(2d)^{d(\lg d)^2} \sim n$
one can show that
$|A|\ge n^{1-\frac{1+\epsilon}{\lg\lg n}}$.
\end{sketch}

We do not use the $d$ that is recommended.
Instead our computer program looks at all possible $d$.
There are not that many $d$'s to check since most
of them will not satisfy the condition that $d$ divides $k$.
Both the value of $d$ recommended, and the ones we use,
are very low and very close together.
For example, for $n=10^{100}$ the recommended $d$ is
around 8 or 9, whereas it turns out that 13 is optimal.

\subsection{3-Free Subsets of Size $n^{1-\frac{3.5\sqrt{2}}{\sqrt{\lg n}}}$: The  Block Method}\label{se:block}

Behrend~\cite{Behrend-1946}
and
Moser~\cite{Moser-1953}
both proved
$\sz(n) \ge  n^{1-\frac{c}{\sqrt{\lg n}}}$,
for some value of $c$.  Behrend proved it first
and with a smaller (hence better) value of $c$, but his proof was
nonconstructive (i.e, the proof does not indicate how to
actually find such a set). Moser's
proof was constructive.
We present Moser's proof here; Behrend's proof
is presented later.

\begin{theorem}\cite{Moser-1953}
For all $n$,
$\sz(n) \ge  n^{1-\frac{3.5\sqrt 2}{\sqrt{\lg n}}}\sim n^{1-\frac{4.2}{\sqrt{\lg n}}}$,
\end{theorem}

\begin{sketch}

Let $r$ be such that
$2^{r(r+1)/2}-1\le n\le 2^{(r+1)(r+2)/2}-1$.
Note that  $r\ge \sqrt{2\lg n}-1$.

We write the numbers in $[n]$ in base 2.
We think of a number as being written in $r$ blocks of bits.
The first (rightmost) block is one bit long.
The second block is two bits long.
The $r$th block is $r$ bits long. Note that the largest possible
number is $r(r+1)/2$ 1's in a row, which is $2^{r(r+1)/2}-1\le n$.
We call these blocks $x_1,\ldots,x_r$.
Let $B_i$ be the number represented by the $i$th block.
The concatenation of two blocks will represent
a number in the natural way.

\noindent
{\bf Example:} We think of $(1001 110 10 1)_2$  as $(1001:110:10:1)$ so
$x_1=(1)_2=1$, $x_2=(10)_2=2$, $x_3=(110)_2=6$, and $x_4=(1001)_2=9$.
We also think of $x_4x_3=(1001110)_2=78$. \\
{\bf End of Example}

The set $A$ is the set of all numbers $x_rx_{r-1}\ldots x_1$ such that
\begin{enumerate}
\item
For $1\le i\le r-2$ the leftmost bit of $x_i$ is 0.
Note that when we add together two numbers in $A$, the
first $r-2$ blocks will add with no carries.
\item
$\sum_{i=1}^{r-2} x_i^2 = x_{r}x_{r-1}$
\end{enumerate}

\noindent
{\bf Example:}
Consider the number $(10110011011000101011010)_2$.
We break this into blocks
to get
$(0000010:110011:01100:0101:011:01:0)_2$.
Note that there are $r=7$ blocks and
the rightmost $r-2=5$ of them all have a 0
as the leftmost bit.
The first 5 blocks, reading from the right, as
base 2 numbers, are
$0=0$,
$01=1$,
$011=3$,
$0101=5$,
$01100=12$.
The leftmost two blocks merged together are
$0000010110011=179$.
Note that $0^2+1^2+3^2+5^2+12^2=179$.
Hence the number
$(10110011011000101011010)_2$
is in $A$. \\
{\bf End of Example}

We omit the proof that $A$ is 3-free, but note that it uses Fact~\ref{fa:twoy}.

How big is $A$?   Once you fill in the first $r-2$ blocks, the content
of the remaining two blocks is determined and will (by an easy calculation)
fit in the allocated $r+(r-1)$ bits.
Hence we need only determine how many ways the first $r-2$ blocks can be filled in.
Let $1\le i\le r-2$. The $i$th block has $i$ places in it, but the leftmost
bit is 0, so we have $i-1$ places to fill, which we can do $2^{i-1}$ ways.
Hence there are $\prod_{i=1}^{r-2} 2^{i-1} = \prod_{i=0}^{r-3} 2^i = 2^{(r-2)(r-3)/2}$.

$(r-2)(r-3) \ge (\sqrt{2\lg n} -3)(\sqrt{2\lg n}-4)= 2\lg n -7\sqrt{2\lg n} +12$

So

$(r-2)(r-3)/2\ge \lg n -3.5\sqrt{2\lg n} +6$

So

$2^{(r-2)(r-3)/2} \ge 2^{\lg n -3.5\sqrt{2\lg n} +6} \sim
n^{1-\frac{3.5\sqrt{2}}{\sqrt{\lg n}}}$
\end{sketch}

The block method allows you to find large 3-free sets quickly.
Table 4 in Appendix III shows the sizes of sets it produces for
rather large values of $n$. We also include an estimate
of the $c$ assuming that the sets are of size
$n^{1-\frac{c}{\sqrt{\lg n}}}$.
The value of $c$ seems pretty steady at between 4 and 5.

\subsection{3-Free Subsets of Size $n^{1-\frac{2\sqrt 2}{\sqrt {\lg n}}}$: The Sphere Methods}\label{se:sphere}

In Sections~\ref{se:B3a}, \ref{se:B5}, \ref{se:KD}, and \ref{se:block}  we presented
constructive methods for finding large 3-free sets of $[n]$ for large $n$.
In this section we present
the Sphere Method, and variants of it, all of which are nonconstructive.
Since the method is nonconstructive it is not obvious how to code it.
However, we have done so and will explain how in this section.
We investigate several variants of the Sphere Method and then compare
them at the end of this section.

\subsubsection{The Sphere Method}

The result and proof in this section are a minor variant of what was
done by Behrend~\cite{Behrend-1946,GRS-1990}.
We will express the number in a base and
put a condition on the representation so that the numbers
do not form a 3-AP.
It will be helpful to think of the numbers as vectors.

\begin{definition}
Let $x,b\in\nat$ and $k=\floor{\log_{b} x}$.
Let $x$ be expressed in base $b$ as
$\sum_{i=0}^k x_i b^i$.
Let $\vec x = (x_0,\ldots,x_k)$
and
$|\vec x|=\sqrt{\sum_{i=0}^k x_i^2}$.
\end{definition}

Behrend used digits $\{0,1,2\ldots,d\}$ in base $2d+1$.
We use digits
$\{-d, -d+1,\ldots, d\}$ in base $4d+1$.
This choice
gives slightly better results since there are more
coefficients to use.
Every number can be represented uniquely in base $4d+1$ with
these coefficients.  There are no carries since
if $a,b\in \{-d,\ldots,d\}$ then $-(4d+1)<a+b<(4d+1)$.

We leave the proof of the following lemma to the reader.

\begin{lemma}\label{le:base}
Let $x=\sum_{i=0}^k x_i (4d+1)^i$,
   $y=\sum_{i=0}^k y_i (4d+1)^i$,
   $z=\sum_{i=0}^k z_i (4d+1)^i$,  where
$-d\le x_i,y_i,z_i\le d$.
Then the following hold.
\begin{enumerate}
\item
$x=y$ iff $(\forall i)[x_i=y_i]$.
\item
If $x+y=2z$ then $(\forall i)[x_i+z_i=2y_i]$
\end{enumerate}
\end{lemma}

The set $A_{d,s,k}$ defined below is the set of all
numbers that, when interpreted as vectors,
have norm $s$ (norm is the square of the length).
These vectors are all on a sphere of radius $\sqrt s$.
We will later impose a condition on $k$ so that
$A_{d,s,k}\subseteq [-n/2,n/2]$.

\begin{definition}
Let $d,s,k\in\nat$.
$$
A_{d,s,k}= \bigg \{ x \st x=\sum_{i=0}^{k-1} x_i (4d+1)^i \wedge (\forall i)[-d\le x_i\le d]
\wedge (|\vec x|^2 = s) \bigg \}
$$
\end{definition}

\begin{definition}
Let $d,s,m\in\nat$.
$$
B_{d,s,k}= \bigg \{ x \st x=\sum_{i=0}^{k-1} x_i (4d+1)^i \wedge (\forall i)[0<x_i\le d]
\wedge (|\vec x|^2 = s) \bigg \}
$$
\end{definition}

\begin{lemma}\label{le:norm}
Let $n,d,s,k\in\nat$.
\begin{enumerate}
\item
$A_{d,s,k}$ is 3-free.
\item
If $n=(4d+1)^k$ then
$A_{d,s,k}\subseteq \{-n/2,\ldots,n/2\}$.
\end{enumerate}
\end{lemma}

\begin{proof}
\noindent
a) Assume, by way of contradiction, that $x,y,z\in A_{d,s,k}$ form a 3-AP.
By Fact~\ref{fa:twoy}, $x+z=2y$.
By Lemma~\ref{le:base}
$(\forall i)[x_i+z_i=2y_i]$.
Therefore $\vec x + \vec z = 2\vec y$, so
$|\vec x + \vec z| = |2\vec y|=2|\vec y|=2\sqrt s$.  Since
$|\vec x|=|\vec z|=\sqrt s$ and $\vec x$ and $\vec z$
are not in the same direction $|\vec x + \vec z|<2\sqrt s$.
This is a contradiction.

\bigskip

\noindent
b) The largest element of $A_{d,s,k}$ is at most
$$\sum_{i=0}^{k-1} d(4d+1)^i = d\sum_{i=0}^{k-1} (4d+1)^i =
\frac{(4d+1)^k-1}{2}=\frac{n-1}{2} \le n/2.$$
Similarly, the smallest element is $\ge -n/2$.
\end{proof}

\begin{lemma}\label{le:AB}
For all $d,s,k$
$$|A_{d,s,k}| = \sum_{m=0}^k {k\choose m} 2^m |B_{d,s,m}|.$$
\end{lemma}

\begin{proof}

Define
$$A_{d,s,k}^m= \bigg \{ x \st x=\sum_{i=0}^{k-1} x_i (4d+1)^i \wedge (\forall i)[-d\le x_i\le d ]$$

$$\wedge (\hbox{ exactly $m$ of the $x_i$'s are nonzero })\wedge (|\vec x|^2 = s) \bigg \} $$

Clearly
$|A_{d,s,k}| = \sum_{m=0}^k |A_{d,s,k}^m|$.

Note that $|A_{d,s,k}^m|$ can be interpreted as first choosing $m$ places to
have non-zero elements (which can be done in $k\choose m$ ways),
then choosing the absolute values of the elements (which can be done
in $|B_{d,s,m}|$ ways) and then choosing the signs
(which can be done in $2^m$ ways). Hence
$|A_{d,s,k}^m|= {k\choose m} 2^m |B_{d,s,m}|$.
So

$$|A_{d,s,k}| = \sum_{m=0}^k {k\choose m} 2^m |B_{d,s,m}|.$$
\end{proof}

\begin{theorem}\label{th:sphere}~
For every $\epsilon$ there exists $n_0$ such that, for all $n\ge n_0$,
$\sz(n)\ge n^{1-\frac{2+\epsilon}{\sqrt{\lg n}}}$.
\item
\end{theorem}

\begin{sketch}

\noindent
Let $d,s,k$ be parameters to be specified later.
We use the set $A_{d,s,k}$ which, by Lemma~\ref{le:norm}, is 3-free.
We seek values of $d,k,s$ such that $|A_{d,s,k}|$ is large and contained in
$[-n/2,n/2]$.
Note that once $k,d$ are set the only possibly values of $s$ are
$\{0,1,\ldots,kd^2\}$.

A calculation shows that if $k\approx \sqrt{\lg n}$ and
$d$ is such that
$n = (4d+1)^k$
then $\bigcup_{s=0}^{kd^2} |A_{d,s,k}|$ is so large that
{\it there exists} a value of $s$ such that
$|A_{d,s,k}| \ge n^{1-\frac{2+\epsilon}{\sqrt{\lg n}}}$.
Note that the proof is nonconstructive in that we do not specify $s$;
we merely show it exists.
\end{sketch}

\noindent
{\bf  Notation}
The method for finding 3-free sets that from Theorem~\ref{th:sphere}
is called {\it The SPHERE method}.

Our proof differs from Behrend's in that we use negative numbers.
The use of negative numbers allowed
us to use one more coordinate, which leads to a slight improvement
in the constant.

The proof recommends using $k\approx \sqrt{\lg n}$ and
$d$ such that $n=(4d+1)^k$.
We optimize over all $k,d,s$.  Table 5, in Appendix IV.
shows, for a variety of values of $n$,
\begin{enumerate}
\item
The size of the largest 3-free set we
could find using optimal values of $d,k,s$
\item
The value of $d$ we used.
\item
The value of $d$ that the proof recommends.
We denote this by $d_{rec}$.
\item
The value of $k$ we used.
\item
The value of $k$ that the proof recommends.
We denote this by $k_{rec}$.
\item
The value of $s$ that we use.
(Note that there is no recommended value of $s$
in the proof; they show that a good value of $s$
exists nonconstructively.)
\end{enumerate}

Note that our value of $k$ is larger than theirs and our
value of $d$ is {\it much} smaller than theirs.
This might make you think that
a more refined proof, using smaller values of $d$,
can yield a better asymptotic result.
However, as we will see, this is unlikely.

Table 6, in Appendix V,  estimates the value of $c$
(assuming the sets are of size $n^{1-\frac{c}{\sqrt{\lg n}}}$)
using a variant
of the SPHERE-method which we will discuss later.
These values seem to be converging as $n$ gets large.
This indicates that the analysis of the SPHERE-method gives
a reasonably tight upper bound on the size of the 3-free sets generated.
Note that this value of $c$ (around 2.54) is better
than that in Table 4 for the Block method (around 4.3).
Also note that even with the optimal values of $d,k,s$
the construction gives the expected asymptotic behavior.
Hence it is unlikely that a different analysis (with larger $k$ and
much smaller $d$) will lead to better asymptotic results.

\subsubsection{Variants of the Sphere Method}

In the proof of Theorem~\ref{th:sphere} we used that
two distinct vectors of size $\sqrt s$ cannot sum to
a vector of size $2\sqrt s$.  We can rephrase this by saying
that if $L=\{s\}$ and $|\vec x|^2, |\vec y|^2, |\vec z|^2 \in L$ then
$\vec x + \vec z \ne 2\vec y$.
We did not use that the vectors were lattice points.
In this section we state and prove some theorems
that lead to slightly better results.

\begin{lemma}\label{le:abc}
Let $s,a,b,c,k\in \nat$.
Let  $\vec x, \vec z, \vec y$ be distinct lattice points in $R^{k+1}$.
Assume that, when interpreted as numbers in some base, $x<y<z$ and they are in arithmetic progression.
Assume that $\vec x + \vec z = 2\vec y$.
Assume
$|\vec x|^2= s+a$,
$|\vec z|^2= s+b$, and
$|\vec y|^2= s+c$.
Let $D=2a+2b-4c$.
The following must all occur.

\begin{enumerate}

\item
$D>0$.

\item
$a+b\le 4c+2s+ 2\sqrt{(s+a)(s+b)}$.

\item
$a+b$ is even.

\item
$D\equiv 0 \pmod 4$ hence (by 1) $D\ge 4$.

\item
$c<\max\{a,b\}$.

\item
There is a representation of $D$ as a sum of squares,
$D=p_0^2 + \cdots + p_f^2$
(in applications of this theorem we can assume
the $p_i$'s are positive),
such that the following hold:

\begin{enumerate}

\item
for all $i$, $p_i$ is even, and

\item
if $f=0$ then $2p_0$ divides $b-a+D$ and one of the following happens:

\begin{enumerate}
\item
$a<b$,  or
\item
$c<a=b$
and
there is an $i$ such that $z_i>0$, $x_i=-z_i<0$ and
$y_i=0$, or
\item
$a>b$
and
there is an $i$ such that $x_i<0$.
\end{enumerate}

\item
$GCD(p_0,\ldots,p_f)$ divides $(b-a+D)/2$.

\item
If $a=b+e$ ($e\ge 0$) then there exists a choice of $f+1$ numbers
$i_1,\ldots, i_{f+1}$ and a choice of pluses and minuses
such that the equation below is satisfied.

$$\sum_{j=0}^f \pm p_{i_j}x_{i_j} = (D+e)/2.$$

\item
If $b=a+e$ ($e\ge 0$) then there exists a choice of $f+1$ numbers
$i_1,\ldots, i_{f+1}$ and a choice of pluses and minuses
such that the equation below is satisfied.

$$\sum_{j=0}^f \pm p_{i_j}z_{i_j} = (D-e)/2.$$

\end{enumerate}
\end{enumerate}
\end{lemma}

\begin{proof}

Throughout the proof let $\vec x = (x_0,\ldots,x_k)$,
$\vec y = (y_0,\ldots,y_k)$, and
$\vec z = (z_0,\ldots,z_k)$.
We will need the following observation. Since $\vec x + \vec z = 2\vec y$, for
every $i$, $x_i+z_i$ is even. Therefore, for every $i$,
$x_i-z_i$ is even and $(x_i-z_i)^2\equiv 0 \pmod 4$.

Look at the parallelogram formed by $\vec 0$, $\vec x$, $\vec z$ and $2\vec y$.
Denote the length of the diagonal from $\vec x$ to $\vec z$
by $\Ld$. Since the sum of the squares of the sides of a parallelogram is the sum of the squares
of the diagonals we have

$$2(s+a) + 2(s+b) = 4(s+c) + \Ld^2.$$

$$2a+2b = 4c + \Ld^2.$$

Note that $D=\Ld^2$.


\bigskip

\noindent
1) If $D=0$ then $\Ld=0$, the parallelogram collapses to a a line, and
$\vec x = \vec z$, a contradiction. If $\Ld < 0$ then the
parallelogram ceases to exist, hence $\vec x + \vec z < 2\vec y$,
a contradiction. So we have $D>0$.

\bigskip

\noindent
2) Since $\Ld$ is the distance between  $\vec x$ and $\vec z$ we have the following.

\[
\begin{array}{rl}
\Ld^2= & \sum_{i=0}^k (x_i-z_i)^2 \cr
     = & \sum_{i=0}^k x_i^2 + \sum_{i=0}^k z_i^2 -2\sum_{i=0}^k x_iz_i\cr
     = & (s+a)+(s+b) - 2\sum_{i=0}^k  x_iz_i\cr
     = &  2s+a+b - 2\sum_{i=0}^k x_iz_i\cr
     = &  2s+a+b - 2|\vec x||\vec z|\cos \theta \hbox{  (where $\theta$ is the angle between $\vec x$ and $\vec z$)}\cr
\end{array}
\]

For all $\theta$, $\cos\theta\ge -1$. Hence

$$\Ld^2 \le 2s+a+b + 2|\vec x||\vec z|= 2s+a+b + 2\sqrt{(s+a)(s+b)}.$$

Hence

$$2a+2b = 4c + \Ld^2 \le 4c + 2s + a + b +2\sqrt{(s+a)(s+b)}\hbox{ and so }$$

$$a+b \le  4c + 2s + 2\sqrt{(s+a)(s+b)}.$$

\bigskip

\noindent
3) Since $|\vec x - \vec z|=\Ld$, $|\vec x - \vec z|^2 = \Ld^2 = D.$

Note that $2(a+b) \equiv 2a+2b-4c \pmod 4$.
This is interesting since $2a+2b-4c=D$.
We now look at $D$ in a different light.

$$D = |\vec x -\vec z|^2 = \sum_{i=0}^k (x_i-z_i)^2.$$

Recall that $x_i-z_i$ is even.

$$\sum_{i=0}^k (x_i-z_i)^2 \equiv 0 \pmod 4.$$

Putting this all together we get
$2(a+b)\equiv 0 \pmod 4$, so $a+b$ is even.

\bigskip

\noindent
4) Since $a+b$ is even, $D=2a+2b-4c = 2(a+b) - 4c \equiv 0 \pmod 4$.
Since $D>0$ and $D\equiv 0 \pmod 4$, we have $D\ge 4$.

\bigskip

\noindent
5) We now show that $c<\max\{a,b\}$. Assume, by way of contradiction,
that $c\ge \max\{a,b\}$.
Then
$D = 2a+2b-4c = 2(a-c) + 2(b-c) \le 0$.
This contradicts $D>0$. Hence $c<\max\{a,b\}$.

\bigskip

\noindent
6) For all $i$, $0\le i\le k$, let $p_i=|x_i-z_i|$. Then $D=\sum_{i=0}^k (x_i-z_i)^2 = \sum_{i=0}^ k p_i^2$.
Let $f+1$ be the number of nonzero terms. Renumber so that, for all $i$, $0\le i\le f$, $p_i^2=(x_i-z_i)^2\ne 0$
(so $x_i-z_i=\pm p_i)$ and for all $i>f$, $(x_i-z_i)^2=0$ (so $x_i=z_i$).
We express all of the $x_i$ in terms of $z_i$ as follows

$\vec z = (z_0,z_1,z_2,\ldots,z_f,z_{f+1},\ldots,z_k).$

$\vec x = (z_0\pm p_0,z_1\pm p_1, z_2\pm p_2,\ldots, z_f\pm p_f, z_{f+1},\ldots,z_k).$

\[
\begin{array}{rl}
b-a=(s+b)-(s+a)= & |\vec z|^2 - |\vec x|^2\cr
               = & \sum_{i=0}^k z_i^2 - (\sum_{i=0}^f (z_i\pm p_i)^2 + \sum_{i=f+1}^k z_i^2)\cr
               = & \sum_{i=0}^f z_i^2 - \sum_{i=0}^f (z_i\pm p_i)^2\cr
               = & \sum_{i=0}^f z_i^2 - (\sum_{i=0}^f z_i^2)+2(\mp p_0z_0 \mp p_1z_1 \mp \cdots \mp p_fz_f)-\sum_{i=0}^f p_i^2\cr
               = & 2(\mp p_0z_0 \mp p_1z_1\mp \cdots \mp p_fz_f) - \sum_{i=0}^f p_i^2 \cr
               = & 2(\mp p_0z_0 \mp p_1z_1\mp \cdots \mp p_fz_f) - D.\cr
\end{array}
\]

We use this to prove 6a, 6b, and 6c.

\bigskip

\noindent
6a) $p_i=\pm (x_i-z_i)$ which is even.

\bigskip

\noindent
6b) Assume $f=0$ (so there is exactly one $i$ such that $x_i\ne z_i$).
Since $b-a=2(\pm z_0p_0)-D$ we know that $2z_0$ divides $D+(b-a)$.

Assume $a\ge b$.
We will later break into the cases $a=b$ and $a>b$.

By the renumbering we can assume that $\vec x$ and $\vec z$ are as follows:

$\vec x = (x_0,x_1,\ldots,x_{k})$,

$\vec z = (x_0\pm p_0,x_1,\ldots,x_{k})$.

Since $x<z$ (as numbers) we have to have that the $\pm$ is actually a $+$.
Hence

$\vec x = (x_0,x_1,\ldots,x_{k})$,

$\vec z = (x_0+p_0,x_1,\ldots,x_{k})$.

Note that

$|\vec x|^2 = \sum_{i=0}^{k} x_i^2 = x_0^2 + \sum_{i=1}^{k} x_i^2$,

$|\vec z|^2 = (x_0+p_0)^2 + \sum_{i=1}^{k} x_i^2$.

Since $a\ge b$ we have $|\vec x|\ge |\vec z|$.
Hence

\[
\begin{array}{rl}
 x_0^2 \ge &  (x_0 +  p_0)^2\cr
 x_0^2 \ge &  x_0^2 + 2x_0p_0 + p_0^2 \cr
     0 \ge & 2x_0p_0 + p_0^2 \cr
     0 \ge & p_0(2x_0 + p_0) \cr
\end{array}
\]

\begin{enumerate}
\item
If $a=b$ then the $\ge$ becomes an $=$ and we get

\[
\begin{array}{rl}
0 = & p_0(2x_0 + p_0)\cr
0 = & 2x_0 + p_0\cr
x_0 = & -p_0/2\cr
z_0 = & x_0 + p_0 = p_0/2\cr
\end{array}
\]

Since $\vec y = (\vec x + \vec z)/2$ we have

$\vec y = (0,x_1,\ldots,x_{k})$ hence

$|\vec y|^2 = \sum_{i=1}^{k} x_i^2 < \sum_{i=0}^{k} x_i^2$.
Therefore $c<a=b$.

\item
If $a>b$ then the $\ge$ becomes a $>$ and we get

\[
\begin{array}{rl}
0 > & p_0(2x_0 + p_0)\cr
0 > & 2x_0 + p_0\cr
\end{array}
\]

Since $p_0>0$ we obtain $x_0<0$.

\end{enumerate}

\bigskip

\noindent
6c) Since $b-a=2(\pm z_0p_0 \cdots \pm z_fp_f)- D$, the Diophantine
equation $\sum_{i=0}^f p_iw_i = (b-a+D)/2$ has a solution in integers.
Hence $GCD(p_0,\ldots,p_f)$ divides $(b-a+D)/2$.
(It is an easy exercise to show that $\sum p_i w_i = E$ has a solution in integers
iff $GCD(p_0,\ldots,p_f)$ divides $E$. See~\cite{IR-1982}, page 15, problems 6, 13, 14
for a guide to how to do this.)

\bigskip

\noindent
6d) Assume $a=b+e$ ($e\ge 0$). Then $|x|^2=|z|^2+e$. By renumbering we can assume

$\vec x = (x_0,x_1,\ldots,x_{k})$,

$\vec z = (x_0\pm p_0,x_1\pm p_1,\ldots,x_{f}\pm p_f, x_{f+1},\ldots, x_k)$.

Since $|x|^2=|z|^2+e$ we have

\[
\begin{array}{rl}
\sum_{i=0}^f x_i^2 = & e+ \sum_{i=0}^f (x_i \pm p_i)^2\cr
\sum_{i=0}^f x_i^2 = & e+ \sum_{i=0}^f x_i^2 + 2\sum_{i=0}^f \pm p_i x_i + \sum_{i=0}^f p_i^2\cr
                  0= &   e+                 2\sum_{i=0}^f \pm p_i x_i + \sum_{i=0}^f p_i^2\cr
                  0= &    e+                2\sum_{i=0}^f \pm p_i x_i + D\cr
                   -(D+e)/2= &  \sum_{i=0}^f \pm p_i x_i \cr
                   (D+e)/2= &  \sum_{i=0}^f \mp p_i x_i \cr
\end{array}
\]

\noindent
6e) Assume $b=a+e$ ($e\ge 0$). Then $|x|^2+e=|z|^2$. By renumbering we can assume

$\vec z = (z_0,z_1,\ldots,z_{k})$,

$\vec x = (z_0\pm p_0,z_1\pm p_1,\ldots,z_{f}\pm p_f, z_{f+1},\ldots, z_k)$.

Since $|x|^2+e=|z|^2$ we have

\[
\begin{array}{rl}
e+\sum_{i=0}^f x_i^2 = &  \sum_{i=0}^f (x_i \pm p_i)^2\cr
e+\sum_{i=0}^f x_i^2 = &  \sum_{i=0}^f x_i^2 + 2\sum_{i=0}^f \pm p_i x_i + \sum_{i=0}^f p_i^2\cr
                   e = &                    2\sum_{i=0}^f \pm p_i x_i + \sum_{i=0}^f p_i^2\cr
                   e = &                  2\sum_{i=0}^f \pm p_i x_i + D\cr
                  (e-D)/2= &  \sum_{i=0}^f \pm p_i x_i \cr
                  (D-e)/2= &  \sum_{i=0}^f \mp p_i x_i \cr
\end{array}
\]

\end{proof}

\begin{definition}
Let $d,s\in\nat$.
Let $C(\vec x)$ be a condition on $\vec x$ (e.g., $(\forall i,j)[|x_i-x_j|\ne 2]$).
Then we define
$$
A_{d,s,k,C}=
\bigg\{
x \st (\exists x_0,\ldots,x_k)
\bigg[x=\sum_{i=0}^k x_i (4d+1)^i \wedge
(\forall i)[-d \le x_i\le d] \wedge |\vec x|^2=s \wedge C(\vec x)
\bigg]
\bigg\}
$$
\end{definition}

\begin{definition}
Let $e,s\in\nat$.
Let $C_1(\vec x)$, $C_2(\vec x)$, $\cdots$, $C_e(\vec x)$ be conditions
on $\vec x$ (e.g., $(\forall i,j)[|x_i-x_j|\ne 2]$).
Then we define
$$
A_{d,s,k,C_1,\ldots,C_e}= A_{d,s,k,C_1\wedge \cdots\wedge C_e}.
$$
\end{definition}

The next two theorems both apply Lemma~\ref{le:abc}
to obtain larger 3-free sets.
The proofs are very similar.

\begin{theorem}\label{th:NZ}
Let $d,k,s\in\nat$.
If $C(\vec x)$ be the condition $(\forall i)[x_i\ne 0]$ then
$A_{d,s,k}\cup A_{d,s+1,k,C}$ is 3-free.
We call this the SPHERE-NZ method. NZ stands for Non-Zero.
\end{theorem}

\begin{proof}
Assume, by way of contradiction, that
$x,y,z\in A_{d,s,k,C}\cup A_{d,s+1,k}$.
Let $a,b,c\in \{0,1\}$ be such that
$|\vec x|^2=s+a$,
$|\vec z|^2=s+b$, and
$|\vec y|^2=s+c$.
There are eight possibilities; however, by Lemma~\ref{le:abc}.5
we can ignore the cases where $c=1$ or $a=b=0$.
For each remaining possibilities we note that either
Lemma~\ref{le:abc} or condition $C$ is violated.

\[
\begin{array}{|c|c|c|c|c|}
\hline
a & b & c & D & \hbox{Reason}\cr
\hline
0 & 1 & 0 & 2 & D\not\equiv 0 \pmod 4\cr
1 & 0 & 0 & 2 & D\not\equiv 0 \pmod 4\cr
1 & 1 & 0 & 4 & \hbox{see below}\cr
\hline
\end{array}
\]

The only case that was not handled
is $a=1$, $b=1$, and $c=0$.
Note that $D=4$.
There is only one way to represent $4$ as a sum of even squares:
$4=2^2$. This corresponds to
Lemma~\ref{le:abc}.6b.
Note that $f=0$ and $a=b$. Hence, by Lemma~\ref{le:abc}.6b.ii,
for some $i$, $y_i=0$.
This contradicts the condition $C$.
\end{proof}

\begin{theorem}\label{th:NN}
Let $d,k,s\in\nat$.
If $C(\vec x)$ is the condition $(\forall i)[x_i\ge 0]$ then
$A_{d,s,k}\cup A_{d,s+1,k,C}$ is 3-free.
We call this the SPHERE-NN method.
NN stands for Non-Negative.
\end{theorem}

\begin{proof}
Assume, by way of contradiction, that
$x,y,z\in A_{d,s,k}\cup A_{d,s+1,k,C}$ form a 3-AP.
Let $a,b,c\in \{0,1\}$ be such that
$|\vec x|^2=s+a$,
$|\vec z|^2=s+b$, and
$|\vec y|^2=s+c$.
There are eight possibilities; however, by Lemma~\ref{le:abc}.5
we can ignore the cases where $c=1$ or $a=b=0$.
For each of remaining possibilities we note that either
Lemma~\ref{le:abc} or condition $C$ is violated.

\[
\begin{array}{|c|c|c|c|c|}
\hline
a & b & c & D & \hbox{Reason}\cr
\hline
0 & 1 & 0 & 2 & D\not\equiv 0 \pmod 4\cr
1 & 0 & 0 & 2 & D\not\equiv 0 \pmod 4\cr
1 & 1 & 0 & 4 & \hbox{see below}\cr
\hline
\end{array}
\]

The only case that was not handled is
 $a=1$, $b=1$, $c=0$.
Note that $D=4$.
There is only one way to represent $4$ as a sum of even squares: $4=2^2$.
This corresponds to
Lemma~\ref{le:abc}.6b.
Note that $f=0$ and $c<a=b$ so we have to have that,
for some $i$, $x_i<0$.
Since $\vec x$ is of length $s+1$
This contradicts condition $C$ on $A_{d,s+1,k,C}$.
\end{proof}

The method of Theorem~\ref{th:NZ}
can be extended; however, this leads to more
complex conditions. We give one more example
and then a general theorem.

\begin{definition}\label{de:condeasy}
Let $a\in\nat$.
The condition $\pm x_1 \pm x_2 \ne a$
is shorthand for the $\wedge$ of the following four
conditions.

\smallskip
\noindent
1) $x_1 + x_2 \ne a$,

\smallskip
\noindent
2) $x_1 - x_2 \ne a$,

\smallskip
\noindent
3) $-x_1 + x_2 \ne a$,

\smallskip
\noindent
4) $-x_1 - x_2 \ne a$.

\end{definition}

\begin{definition}\label{de:cond}
Let $a,f\in\nat$.
The condition $\pm x_1 \pm x_2 \pm \cdots \pm x_f \ne a$
is shorthand for the $\wedge$ of the following $2^f$ conditions

\smallskip
\noindent
1) $x_1 + x_2 + \cdots + x_f \ne a$,

\smallskip
\noindent
2) $x_1 + x_2 + \cdots + x_{f-1} - x_f \ne a$,

\smallskip
\quad \vdots \qquad \qquad \qquad \qquad\qquad \vdots

\smallskip
\noindent
$2^k$) $-x_1 - x_2 - \cdots - x_{f-1} - x_f \ne a$.

\end{definition}

\begin{theorem}\label{th:COND}
Let $d,k,s\in\nat$.
Let $C_1(\vec x)$ be the condition
$(\forall i)[x_i\ne 0]$.
Let $C_2(\vec x)$ be the condition
$(\forall i,j)[i\ne j \implies \pm x_i \pm x_j \ne 2]$.
The set $A_{d,s,k}\cup A_{d,s+1,k,C_1} \cup A_{d,s+2,C_1,C_2}$ is 3-free.
\end{theorem}

\begin{proof}
Assume, by way of contradiction, that
$$x,y,z\in A_{d,s,k}\cup A_{d,s+1,k,C_1}\cup A_{d,s+2,k,C_1,C_2}$$
form a 3-AP.
Let $a,b,c\in \{0,1,2\}$ be such that
$|\vec x|^2=s+a$,
$|\vec z|^2=s+b$, and
$|\vec y|^2=s+c$.
There are 27 possibilities  for $a,b,c$.
By Lemma~\ref{le:abc}.5 we need not consider any case where $c=2$.
By Theorem~\ref{th:NZ}, we need not consider any case
with $a,b,c \in \{0,1\}$.
By Theorem~\ref{th:NZ}, with $s+1$ instead of $s$, we need not
consider any case with $a,b,c \in \{1,2\}$.
This leaves the following cases:

\[
\begin{array}{|c|c|c|c|c|}
\hline
a & b & c & D & \hbox{Reason}\cr
\hline
0 & 2 & 0 & 4 & \hbox{See Below}\cr
0 & 2 & 1 & 0 & D\le 0\hbox{ Lemma~\ref{le:abc}.1}\cr
1 & 2 & 0 & 6 & D\not\equiv 0 \pmod 4\hbox{ Lemma~\ref{le:abc}.4} \cr
2 & 2 & 0 & 8 & \hbox{See Below}\cr
\hline
\end{array}
\]

We now consider the two cases not covered in the chart above.

\noindent
{\bf Case 1:}
$a=0$, $b=2$, and $c=0$.
In this case $D=4$.
There is only one way to represent 4 as a sum of even squares.
$4=2^2$. Hence $f=0$ and $p_0=2$.
Note that $D+(b-a)=6$ and $2p_0=4$. By Lemma~\ref{le:abc}.c
$2p_0$ divides $D+(b-a)$.
This means 4 divides 6, a contradiction.
\bigskip

\noindent
{\bf Case 2:}
$a=2$, $b=2$, and $c=0$.
In this case $D=8$.
There is only one way to represent $8$ as a sum of even squares.
$8=2^2+ 2^2$. Hence $p_0=p_1=2$.
By Lemma~\ref{le:abc}.6d we have that there exists $i,j$
such that $\pm 2x_i \pm 2x_j = \frac{1}{2} 8 = 4$.
Hence $\pm x_i \pm x_j = 2$.
This violates condition $C_2$.
\end{proof}

\begin{theorem}\label{th:CONDGEN}
Let $d,k,s,g\in\nat$.
There exist conditions $E_1,\ldots,E_k$ which are conjunctions of conditions of the type in
Definition~\ref{de:cond},
such that the following set is 3-free.
$$A_{d,s,k}\cup A_{d,s+1,k,E_1}\cup \cdots \cup  A_{d,s+k,k,E_1,\cdots,E_{k}}.$$

\end{theorem}

\begin{proof}

We proof this by induction on $k$.
For $k=1,2$ we know the theorem is true by Theorems~\ref{th:NZ} and~\ref{th:COND}.
We assume the theorem is true at $k-1$
with conditions $E_1,\ldots,E_{k-1}$ and come up
with condition $E_k$.
We do a constructive induction in that we do not know $E_k$
originally, but will know it at the end of the proof.

Assume, by way of contradiction, that
$$
x,y,z\in
A_{d,s,k}\cup A_{d,s+1,k,E_1}\cup \cdots \cup  A_{d,s+k,k,E_1,\cdots,E_{k}}
$$
form a 3-AP.
Let $a,b,c\in \{0,1,2,3,\ldots,k\}$ be such that
$|\vec x|^2=s+a$,
$|\vec z|^2=s+b$, and
$|\vec y|^2=s+c$.
There are $k^3$ possibilities  for $a,b,c$.
By Lemma~\ref{le:abc}.5 we need not consider any case where $c=k$.
By the induction hypothesis we need not consider any case
with $a,b,c \in \{0,1,\ldots,k-1\}$.
By the induction hypothesis, with $s+1$ instead of $s$, we need not
consider any case with $a,b,c \in \{1,2,\ldots,k\}$.
Hence we need only consider the case where one of $a,b$ is $k$,
and one of $a,b,c$ is 0.

Form a table similar to those in Theorem~\ref{th:NZ} and Theorem~\ref{th:COND}.
The entries fall into several categories.
\begin{enumerate}
\item
$D\le 0$.
\item
$D\not\equiv 0 \pmod 4$.
\item
Every decomposition of $D$ into even squares, $D=p_0^2+\cdots+p_f^2$,
has the property that $GCD(p_0,\ldots,p_f)$ does not divide $(b-a-D)/2$.
\end{enumerate}
All of these categories contradiction Lemma~\ref{le:abc}.6c.

Let $(a,b,c,D)$ be a row that does not lead to contradiction.
There are two cases: $a\ge b$ and $a<b$.

Assume $a\ge b$.
Let  $e$ be such that $a=b+e$.
Since at least one of $a,b$ is $k$, we have $a=k$.
For every decomposition of $D$ into even squares such that
$GCD(p_0,\ldots,p_f)$ does not divide $(b-a+D)/2$
we have, by Lemma~\ref{le:abc}.6d, there exists
a choice of $f+1$ numbers
$i_1,\ldots, i_{f+1}$ and a choice of pluses and minuses
such that the equation below is satisfied.

$$\sum_{j=0}^f \pm p_{i_j}x_{i_j} = (D+e)/2.$$

Take this equation and make its negation into a condition.
Note that this is a condition on $\vec x$, and
$|x|=s+a=s+k$, so the condition is on $A_{d,s,k}$.

Assume $a<b$.
Let  $e$ be such that $b=a+e$.
Since at least one of $a,b$ is $k$ we have $b=k$.
For every decomposition of $D$ into even squares such that
$GCD(p_0,\ldots,p_f)$ does not divide $(b-a+D)/2$
we have, by Lemma~\ref{le:abc}.6d, there exists
a choice of $f+1$ numbers
$i_1,\ldots, i_{f+1}$ and a choice of pluses and minuses
such that the equation below is satisfied.

$$\sum_{j=0}^f \pm p_{i_j}z_{i_j} = (D+e)/2.$$

Take this equation and make its negation into a condition.
Note that this is a condition on $\vec z$, and
$|z|=s+b=s+k$, so the condition is on $A_{d,s,k}$.

There may be many conditions, each decomposition
of $D$ into even squares may lead to one.
Let $E_k$ be the conjunction of all of these conditions over
all of the rows.
\end{proof}

We can apply Lemma~\ref{le:abc} to the case where
we have one of the vectors much larger.

\begin{theorem}\label{th:spherefar}
Let $d,s,k\in\nat$.
Let $C$ be the condition $x_i\ge 0$.
The set $A_{d,s,k}\cup A_{d,10s,k,C}$ is 3-free.
We call this the {\it SPHERE-FAR method}.
\end{theorem}

\begin{proof}
Assume, by way of contradiction, that
$x,y,z\in A_{d,s,k}\cup A_{d,10s,k,C}$
form a 3-AP.
Let $a,b,c\in \{0,9s\}$ be such that
$|\vec x|^2=s+a$,
$|\vec z|^2=s+b$, and
$|\vec y|^2=s+c$.
There are eight cases to consider.
By Lemma~\ref{le:abc}.5 we need not consider any case where $c=9s$.
We use Lemma~\ref{le:abc}.2 in the table below.

\[
\begin{array}{ccccccc}
a & b   & c  & D    & a+b & 4c+2s+2\sqrt{(s+a)(s+b)} &  \hbox{Reason}\cr
0 & 9s  & 0  & 18s  &  9s & (2+\sqrt{10})s           &  a+b > 4c+2s+2\sqrt{(s+a)(s+b)} \cr
9s & 0  & 0  & 18s  & 9s  & (2+2\sqrt{10})s          &  a+b > 4c+2s+2\sqrt{(s+a)(s+b)}  \cr
9s & 9s & 0  & 36s  & 18s & 22s                      &  \hbox{See Below} \cr
\end{array}
\]

The only case to consider is $a=b=9s$ and $c=0$.
Note that

\[
\begin{array}{rl}
2 \vec y = & \vec x + \vec z \cr
|2 \vec y|^2 = & |\vec x + \vec z|^2 \cr
|(2y_0,\ldots,2y_k)|^2 = & |(x_0+z_0),\ldots,(x_k+z_k)|^2 \cr
\sum_{i=0}^k 2y_i^2 = & \sum_{i=0}^k (x_i+z_i)^2\cr
\sum_{i=0}^k 2y_i^2 = & \sum_{i=0}^k x_i^2 +2x_iz_i + z_i^2\cr
2\sum_{i=0}^k y_i^2 = & (\sum_{i=0}^k x_i^2) +(2\sum_{i=0}^k x_iz_i) + \sum_{i=0}^k z_i^2 \cr
2s = & s+9s + 2(\sum_{i=0}^k x_iz_i) + s+9s \cr
2s = & 20s + 2(\sum_{i=0}^k x_iz_i) \cr
-9s = & \sum_{i=0}^k x_iz_i \cr
\end{array}
\]

Note that since $x_i,z_i\ge 0$ we have that the right hand side is $\ge 0$.
But the left hand side is $<0$. This is a contradiction.
\end{proof}

\begin{note}
In the proof that $A_{d,s,k}\cup A_{d,10s,k,C}$ we used the value `10' twice.
In the first two rows of the table we needed the value to be as high as 10.
In the third row of the table we could have used 2.
\end{note}

\subsubsection{Using the Sphere Method}

In this section, for ease of exposition, we discuss how to use the Sphere
method if you were only dealing with nonnegative numbers.
The methods discussed can be easily adjusted for the case where
numbers can be positive or negative.

The next theorem shows how to
find the optimal $(d,s)$ pair quickly if you have {\it a priori} bounds on what
you are looking for.

\begin{theorem}\label{th:dks}
Let
$N(d,k,s)=|\{ (x_0,\ldots,x_k) \st 0\le x_i\le d\wedge \sum_{i=0}^k x_i^2 = s\}|.$
\begin{enumerate}
\item
Let $D,K,S\in \nat$.
We can determine the values of $N(d,k,s)$ for all $d,k,s$ such that
$0\le d\le D$, $0\le k\le K$, and $0\le s\le S$ in
time $O(DK^2S)$.
\item
There is an algorithm that will, given $n$, find
the optimal $(d,s)$ pair in time $O(n^3\log^3 n)$.
\end{enumerate}
\end{theorem}

\begin{proof}
\smallskip
\noindent
1) Note that for all $k$ and for all $s\ne 0$,
\[
\begin{array}{rl}
N(0,k,0)=&1;\cr
N(0,k,s)=&0.\cr
\end{array}
\]

Let
$$N_j(d,k,s)=|\{ (x_0,\ldots,x_k) \st 0\le x_i\le d \wedge \sum_{i=0}^k x_i^2 = s\wedge
\hbox{ exactly $j$ of the components are $d$ }\}|.$$
Note that since the $j$ $d$'s could be in any of $k \choose j$ places,
and the remaining $k-j$ components must add up to $s-jd^2$, using numbers
$\le d-1$,
we have
$N_j(d,k,s)={k \choose j }N(d-1,k-j,s-jd^2)$.
Hence

$$N(d,k,s)=\sum_{j=0}^k N_j(d,k,s) = \sum_{j=0}^k {k \choose j }N(d-1,k-j,s-jd^2).$$

To use this we first compute the ${k\choose j}$ for $1\le k\le K$ and $0\le j\le k$.
This can be done using the recurrence
${k \choose j } = {k-1 \choose j} + {k-1 \choose j-1 }$ and dynamic programming.
This will take $O(K^2)$ time.
Using these numbers and the recurrence for $N(d,k,s)$,
we can easily write a dynamic program that
runs in time $O(K^2+DK^2S)=O(DK^2S)$.

\smallskip
\noindent
2) The optimal value of $d$ is $\le n$.
The largest $k$ can be is $\log n$.
The largest $s$ can be is $d^2k \le n^2\log n$.
Hence, applying the first part of this theorem,
we can compute the optimal $(d,s)$ pair in $O(n^3\log^3 n)$ steps.
\end{proof}

Table 5, in Appendix IV,  suggests that the optimal $d$ is actually $\le O(\log n)$.
If this is the case then the run time can be reduced to $O(\log^6 n)$.

When we have a particular $n$ in mind, $N(d, k, s)$ may overcount,
since some of the numbers counted may be greater in magnitude than
$n$.  Denote by $A_{d,s,k,n}$ the set of numbers in
$A_{d,s,k}$ that are no greater in absolute value than $n$.
In order to compute $\mid A_{d,s,k,n} \mid$ we must do some extra work.
Denote by $L_{d,s,k,n}$ the set of integers in $A_{d,s,k}$ that are greater
than $n$, so that $L_{d,s,k,n} = A_{d,s,k} - A_{d,s,k,n}$.
We next partition $L_{d,s,k,n}$ according to the most significant digit
in which its members differ from $n$.
Write $n$ in base $(4d + 1)$ as $n_1, \ldots , n_k$ and, for $0 \le i < k$
let $L_{d,s,k,i}$ be the subset of $L_{d,s,k,n}$ whose elements
start with $n_1\ldots n_i$ (note that if any of $n_1, \ldots, n_i$ has
absolute value greater than $d$ then $L_{d,s,k,n,i} = \emptyset$).
Now the $(i+1)$st digit of each element
of $L_{d,s,k,n,i}$ must be between $n_{i+1} + 1$ and $d$ inclusive,
but the remaining digits are subject only to the constraint that the sum
of the squares is $s$.  If you consider only those remaining digits,
they represent a number with $k-i-1$ digits, each less than $d$ in
absolute value whose squares add up to $s - n_1^2 - \cdots n_i^2 - x_{i+1}^2$.
Considering all possible values of $x_{i+1}$ leads to an expression for
the size of any nonempty $L_{d,s,k,n,i}$:
$$
\mid L_{d,s,k,n,i} \mid = \sum_{x = n_{i+1} +1}^d N(d, k - i - 1, s - \sum_{j=1}^i n_j^2 - x^2).
$$
Summing this quantity over all possible $i$ gives us $\mid L_{d,s,k,n} \mid$
and hence $A_{d,s,k,n}$.

The calculations of $N(d,k,s)$ and $|A_{d,s,k,n}|$ can be modified easily
for the cases in which negative digits and/or zero digits are not allowed.

By Theorem~\ref{th:NZ} the set $A_{d,s,k,C}\union A_{d,s+1,k}$ is 3-free.
The condition $C$ is simple, so an easy modification of
Theorem~\ref{th:dks} gives us a way to find the the value of $s$ that optimizes
$|A_{d,s,k,C}|+|A_{d,s+1,k}|$ quickly.
Similarly for Theorem~\ref{th:NN}.
One can derive more complex conditions and hence larger 3-free sets.
Our initial attempts at this yielded very little gain; hence
we did not pursue it any further. Nevertheless, for these more
complex conditions one can still compute the sizes quickly
as we show below. This technique may be useful to later researchers.
We present them for the case where all the digits are $\ge 0$
for ease of exposition.

$$
N(d,k,s,C)=
|\{ (x_0,\ldots,x_k) \st 0\le x_i\le d \wedge \sum_{i=0}^k x_i^2 = s\wedge  C(\vec x)\}|.
$$

In order to compute this we need to keep track of which elements we are
already using and which ones are forbidden.
We let $U$ be the {\it multiset} of elements already being used
and $F$ be the {\it set} of elements forbidden from being used.
We also let $k_0$ be the original value of $k$ we started with.
Note that $C$ will be a $k_0$-ary predicate.

Let
$N(d,k_0,k,s, C, U,F)$ be
all
$(x_0,\ldots,x_k)$
such that
\begin{enumerate}
\item
for all $i$, $0\le x_i\le d$ and  $x_i\notin F$,
\item
$\sum_{i=0}^k x_i^2 = s$
\item
Let $U'= U \cup \{x_1,\ldots,x_k\}$ (a multiset).
For any vector $\vec {u'}$ of elements in $U'$,
we have $C(\vec {u'})$.
\end{enumerate}

Let
$N_j(d,k_0,k,s, C, U, F)$ be the subset of $N(d,k,s, C, U, F)$
such that
exactly $j$ of the components are $d$.
We now define
$N_j(d,k_0,k,s, C, U, F)$ with a recurrence.

\begin{enumerate}
\item
Assume $d\in F$. Then
$$
N_j(d,k_0,k,s, C, U, F)= \cases { 0 & if $j\ge 1$;\cr
                                    N(d-1,k_0,k,s, C, U, F) & if $j=0$.\cr
}
$$
\item
Assume $j=0$. Then
$N_0(d,k_0,k,s, C, U, F)= N(d-1,k_0,k,s, C, U, F)$.

\item
Assume $d\notin F$ and $j\ge 1$.
\begin{enumerate}
\item
$U_j=U\union \{d,d,\ldots,d\}$ ($j$ times)
\item
$F_j=F \union
\{ f \st (\exists u_0,\ldots,u_{\ell-1},u_{\ell+1},\ldots,u_{k_0}\in U_j)
                              [\neg C(u_0,\ldots,u_{\ell-1},f,u_{\ell+1},\ldots,u_{k_0}] \}$
(This is the only time we use $k_0$.)
\item
$N_j(d,k_0,k,s, C, U, F)={k \choose j }N(d-1,k_0,k-j,s-jd^2, C, U_j, F_j)$.
\end{enumerate}
\end{enumerate}

In summary, $N(d,k_0,k,s,C,U,F)$
is
\begin{enumerate}
\item
$N(d-1,k_0,k,s,C,U,F)$ if $d\in F$, and
\item
$N(d-1,k_0,k,s,C,U,F)+\sum_{j=1}^k {k \choose j }N(d-1,k_0,k-j,s-jd^2, C, U_j, F_j)$ if $d\notin F$.
\end{enumerate}

\subsubsection{Comparing the Sphere Methods}

We now present comparisons between the different Sphere Methods.

In Table 7, in Appendix VI,  we use the following notation.
\begin{enumerate}
\item
SPHERE denotes $|A_{d,s,k}|$.
\item
SPHERE-NZ denotes $|A_{d,s,k} \union A_{d,s+1,k,C}|$
where $C$ is the condition that all digits are nonzero.
Informally, we will take the vectors on two spheres: on that is $s$ away
from the origin, and one that is $s+1$ away from the origin;
however, the coordinates in the one that is $s+1$ away have to all be nonzero.
\item
SPHERE-$NN$ denotes $|A_{d,s,k,C} \union A_{d,s+1,k,C}|$
where $C$ is the condition that all digits are nonnegative.
We maximize over all $d,s,k$.
\end{enumerate}

In all three cases
we maximize over all $d,s,k$ such that
$k=\floor{\log_{4d+1} n}-1$.
The data indicates that the SPHERE-NZ method is the best one. 

\section{Comparing All the Methods}\label{se:3free}

Tables 8-13 in Appendix VII compares the size of 3-sets generated
by most of the methods of this paper for $n=10^1,\ldots,10^{65}$.
Then from $n=10^{66}$ to $n=10^{100}$ we show what happens for
all methods except the SPHERE method which is too slow to run
for those values.
We abbreviate the names of the methods as
B3 (Base 3), B5 (Base 5), KD (KD), BL (Block) and SP (SPHERE-NN)
method).

Here are some observations

\begin{enumerate}
\item
For all $n$ on our table either Base 3 or Sphere is the best
method. There may be a particular number which, due
to its representation in base 5, the Base 5 method
does better than either Base 3 or Sphere.
For $10^9 \le n \le 10^{65}$, the Sphere method is producing
larger 3-free sets than any other method.  We stopped at
$10^{65}$ since at that point the Sphere method took too much
time.
Given this evidence and that asymptotically the Sphere method
produces larger 3-free sets, we suspect that for $n\ge 10^9$
the Sphere method really does produce larger 3-free sets of $[n]$
than the other methods.
\item
BL is better than KD asymptotically but
this is the pair that takes the longest
to settle down.  They switch back and fourth
quite a bit. This is because BL is particularly
sensitive to what type of number $n$ is.
When $n\ge 10^{90}$ the table suggests that BL produces larger
3-free sets and will from then on.
Given this evidence and that asymptotically BL is better than KD,
we suspect that for $n\ge 10^{90}$ BL really does produces larger sets than KD.
More generally, we suspect that for $n\ge 10^{90}$ the asymptotic
behavior will match the empirical behavior for
all the methods with regard to which one produces the
largest, second largest, etc 3-free sets.
\end{enumerate}

\section{Using the Asymptotic Literature for Upper Bounds}\label{se:upper}

Roth~\cite{GRS-1990,Roth-1953} showed the following:
for every $\lambda>0$ there exists
$n_0$ such that, for every $n\ge n_0$,
for every $A\subseteq [n]$, if $|A|\ge \lambda n$
then $A$ has an arithmetic progression of length 3.

The proof as presented in~\cite{GRS-1990} can help us obtain
upper bounds.
They actually prove the following:

\begin{theorem}\label{th:uppersz}
Let $c$ be such that $0<c<1$ and let $m\in\nat$.
Assume that $\sz(2m+1)\le c(2m+1)$.
Assume that $N,M\in\nat$ and $\epsilon>0$
satisfy the following:

$$(\frac{m^2N}{2M^2}   + 4\epsilon N  + 4mM + 1)(c+\epsilon) < c(c-\epsilon)N$$

Then $\sz(N) \le (c-\epsilon)N$.
\end{theorem}

Theorem~\ref{th:good} and the comments after it yield an elementary
method to obtain upper bounds on $\sz(N)$.
Theorem~\ref{th:uppersz} yields a more sophisticated method; however, is it better?
Tables 14-16 in Appendix VIII shows that, for large values of $N$,  it is better.
We take $m=50$.  We know $\sz(101)\le 0.26733\times 101$
hence we can take $c=0.26733$.
We take $M$ to be values between 119 and 1000.
For each of these values we find the minimal $N$ such that there is
an $\epsilon$ such
that the theorem can be applied.
We then note the percent improvement over the elementary method.

A more careful analysis of Roth's theorem (or alternative proofs of it)
may yield better bounds. Our interest  would be to get bounds
that work for lower numbers.

\section{Future Directions}

\begin{enumerate}
\item
We believe that using the methods in this paper and current
technology, the value of $\sz(200)$ can be obtained.
We would like to develop techniques that get us much further than that.
\item
A more careful examination of upper bounds in the literature, namely
a detailed look at the results of Roth, Szemeredi, Health-Brown, and
Bourgain mentioned earlier, may lead to better upper bounds.
\item
This paper has dealt with 3-AP's.
Similar work could be carried out for $k$-AP's.
Not much is known about them; however,
\cite{Wagstaff-1972} is a good start.
\end{enumerate}

\section{Acknowledgments}

We would like to thank
Jon Chapin,
Walid Gomaa,
Andre Utis,
and the referees
for proofreading and commentary.
We would also like to thank Ryan Farrell, Tasha Innis, and Howard Karloff for
help on some of the integer and linear programming methods, and
Robert Kleinberg for discussions of matrix multiplication and 3-free sets.
The first author would like to thank NSF grant CCR-01-05413.

\section{Appendix I: Comparison to Known Results}

There are several websites that contain results similar to ours:

\begin{itemize}
\item
http://www.math.uni.wroc.pl/$\tilde{}$jwr/non-ave/index.htm
\item
http://www.research.att.com/$\tilde{}$njas/sequences/A065825
\item
http://www.research.att.com/$\tilde{}$njas/sequences/A003002
\end{itemize}

The first one is a website about {\it Nonaveraging sets search}.
A {\it nonaveraging set} is what we have been calling a
3-free set.  They study the problem in a different way.
\begin{definition}
For $m\in \nat$ $a(m)$ 
is the least number so that there is a nonaveraging subset of $\{1,\ldots,a(m)\}$.
\end{definition}

The following are easily verified.

\begin{fact}
\item
$\sz(a(m))\ge m$.
\item
$\sz(n)\ge m$ iff $a(m)\le n$.
Hence large 3-free sets yield upper bounds
on $a(m)$ and vice-versa.
\item
If $\sz(n)<m$ then $a(m)>n$.
\item
If $\sz(n)=m-1$ and $\sz(n+1)=m$ then $a(m)=n$.
\end{fact}

At the website they have 
exact values for $a(m)$ for $m\le 35$.
and upper bounds for $a(m)$ (hence 3-free sets) for
$m\le 1024$.
They have $a(35)=150$ which yields $\sz(150)=35$.

Our table yields the following new results:
$a(37)=163$, 
$a(38)=167$, 
$a(39)=169$, 
$a(40)=174$,
and $a(42)\le 204$ (they had 205).

We summarize the difference between our data and
the websites above:

\begin{enumerate}
\item
Our table yields the following new results stated in their terms: 
$a(37)=163$, $a(38)=167$, $a(39)=169$, $a(40)=174$,
$a(42)\le 204$, $a(45)\le 227$, and $a(46)\le 233$.

\item
Our table yields the following new results stated in our terms:
\begin{enumerate}
\item
Before our paper $\sz(n)$ was known for $n=1,\ldots,150$.
Our paper has extended this to $n=151,\ldots, \lastnum$.
\item
$\sz(204)\ge 42$,
$\sz(227)\ge 45$, and
$\sz(233)\ge 46$.
\end{enumerate}

\item
For several values of $n$ over 1000, they have obtained lower bounds
on $\sz(n)$ (that is, large 3-free sets) that we have
not been able to obtain.

\item
The second website is the entry on $a(n)$ in the Online Encyclopedia.
Currently the first website has the most current results.
The third website is the entry in the Online Encyclopedia of $\sz(n)$.
It only has values up to $n=53$. 
\end{enumerate}

\vfill\eject

\section{Appendix II: Tables for Small $n$}

\begin{table}[htbp]
\begin{center}
\begin{tabular}{cc@{\hspace{0.3in}}|@{\hspace{0.3in}}cc@{\hspace{0.3in}}|@{\hspace{0.3in}}cc@{\hspace{0.3in}}|@{\hspace{0.3in}}cc}
$n$ & $\sz(n)$ & $n$ & $\sz(n)$ & $n$ & $\sz(n)$ & $n$ & $\sz(n)$ \cr
\hline
 1& 1 & 26&11 & 51&17 & 76&22 \cr
 2& 2 & 27&11 & 52&17 & 77&22 \cr
 3& 2 & 28&11 & 53&17 & 78&22 \cr
 4& 3 & 29&11 & 54&18 & 79&22 \cr
 5& 4 & 30&12 & 55&18 & 80&22 \cr
 6& 4 & 31&12 & 56&18 & 81&22 \cr
 7& 4 & 32&13 & 57&18 & 82&23 \cr
 8& 4 & 33&13 & 58&19 & 83&23 \cr
 9& 5 & 34&13 & 59&19 & 84&24 \cr
10& 5 & 35&13 & 60&19 & 85&24 \cr
11& 6 & 36&14 & 61&19 & 86&24 \cr
12& 6 & 37&14 & 62&19 & 87&24 \cr
13& 7 & 38&14 & 63&20 & 88&24 \cr
14& 8 & 39&14 & 64&20 & 89&24 \cr
15& 8 & 40&15 & 65&20 & 90&24 \cr
16& 8 & 41&16 & 66&20 & 91&24 \cr
17& 8 & 42&16 & 67&20 & 92&25 \cr
18& 8 & 43&16 & 68&20 & 93&25 \cr
19& 8 & 44&16 & 69&20 & 94&25 \cr
20& 9 & 45&16 & 70&20 & 95&26 \cr
21& 9 & 46&16 & 71&21 & 96&26 \cr
22& 9 & 47&16 & 72&21 & 97&26 \cr
23& 9 & 48&16 & 73&21 & 98&26 \cr
24&10 & 49&16 & 74&22 & 99&26 \cr
25&10 & 50&16 & 75&22 & 100&27\cr
\end{tabular}
\caption{Values of $\sz(n)$; 1-100 found by Intelligent Backtracking}
\end{center}
\end{table}

\vfill\eject

\begin{table}[htbp]
\begin{center}
\begin{tabular}{cc@{\hspace{0.3in}}|@{\hspace{0.3in}}cc@{\hspace{0.3in}}|@{\hspace{0.3in}}cc@{\hspace{0.3in}}|@{\hspace{0.3in}}cc}
$n$ & $\sz(n)$ & $n$ & $\sz(n)$ & $n$ & $\sz(n)$ & $n$ & $\sz(n)$ \cr
\hline
101 & 27 & 126 & 32 & 151 & 35 & 176 & 40 \cr
102 & 27 & 127 & 32 & 152 & 35 & 177 & 40 \cr
103 & 27 & 128 & 32 & 153 & 35 & 178 & 40 \cr
104 & 28 & 129 & 32 & 154 & 35 & 179 & 40 \cr
105 & 28 & 130 & 32 & 155 & 35 & 180 & 40 \cr
106 & 28 & 131 & 32 & 156 & 35 & 181 & 40 \cr
107 & 28 & 132 & 32 & 157 & 36 & 182 & 40 \cr
108 & 28 & 133 & 32 & 158 & 36 & 183 & 40 \cr
109 & 28 & 134 & 32 & 159 & 36 & 184 & 40 \cr
110 & 28 & 135 & 32 & 160 & 36 & 185 & 40 \cr
111 & 29 & 136 & 32 & 161 & 36 & 186 & 40 \cr
112 & 29 & 137 & 33 & 162 & 36 &  &  \cr
113 & 29 & 138 & 33 & 163 & 37 &  &  \cr
114 & 30 & 139 & 33 & 164 & 37 &  &  \cr
115 & 30 & 140 & 33 & 165 & 38 &  &  \cr
116 & 30 & 141 & 33 & 166 & 38 &  &  \cr
117 & 30 & 142 & 33 & 167 & 38 &  &  \cr
118 & 30 & 143 & 33 & 168 & 38 &  &  \cr
119 & 30 & 144 & 33 & 169 & 39 &  &  \cr
120 & 30 & 145 & 34 & 170 & 39 &  &  \cr
121 & 31 & 146 & 34 & 171 & 39 &  &  \cr
122 & 32 & 147 & 34 & 172 & 39 &  &  \cr
123 & 32 & 148 & 34 & 173 & 39 &  &  \cr
124 & 32 & 149 & 34 & 174 & 40 &  &  \cr
125 & 32 & 150 & 35 & 175 & 40 &  &  \cr
\end{tabular}
\caption{Values of $\sz(n)$; 101-186 found by Intelligent Backtracking}
\end{center}
\end{table}

\begin{table}[htbp]
\begin{center}
\begin{tabular}{ccc@{\hspace{0.3in}}|@{\hspace{0.3in}}ccc@{\hspace{0.3in}}|@{\hspace{0.3in}}ccc}
$n$ & low & high & $n$ & low & high & $n$ & low & high \cr
\hline

187 & 40 & 41 & 212 & 43 & 50 & 237 & 46 & 57 \cr
188 & 40 & 42 & 213 & 43 & 51 & 238 & 46 & 57 \cr
189 & 40 & 42 & 214 & 43 & 51 & 239 & 47 & 57 \cr
190 & 40 & 43 & 215 & 44 & 51 & 240 & 47 & 58 \cr
191 & 40 & 44 & 216 & 44 & 51 & 241 & 47 & 58 \cr
192 & 40 & 44 & 217 & 44 & 51 & 242 & 47 & 58 \cr
193 & 40 & 44 & 218 & 44 & 51 & 243 & 47 & 58 \cr
194 & 41 & 44 & 219 & 44 & 51 & 244 & 47 & 58 \cr
195 & 41 & 45 & 220 & 44 & 52 & 245 & 47 & 58 \cr
196 & 41 & 45 & 221 & 44 & 52 & 246 & 47 & 59 \cr
197 & 41 & 46 & 222 & 44 & 52 & 247 & 48 & 59 \cr
198 & 41 & 46 & 223 & 44 & 53 & 248 & 48 & 59 \cr
199 & 41 & 47 & 224 & 44 & 53 & 249 & 48 & 60 \cr
200 & 41 & 47 & 225 & 44 & 54 & 250 & 48 & 60 \cr
201 & 41 & 48 & 226 & 44 & 54 &  & &  \cr
202 & 41 & 48 & 227 & 45 & 55 &  & &  \cr
203 & 41 & 48 & 228 & 45 & 55 &  & &  \cr
204 & 42 & 48 & 229 & 45 & 55 &  & &  \cr
205 & 42 & 48 & 230 & 45 & 56 &  & &  \cr
206 & 42 & 49 & 231 & 45 & 56 &  & &  \cr
207 & 42 & 49 & 232 & 45 & 56 &  & &  \cr
208 & 42 & 49 & 233 & 46 & 56 &  & &  \cr
209 & 43 & 49 & 234 & 46 & 56 &  & &  \cr
210 & 43 & 49 & 235 & 46 & 56 &  & &  \cr
211 & 43 & 50 & 236 & 46 & 56 &  & &  \cr
\end{tabular}
\caption{Upper and Lower Bounds on $\sz(n)$}
\end{center}
\end{table}

\vfill\eject

\section{Appendix III: The value of $c$ for the Block Method}

\begin{table}[htbp]
\begin{center}
\begin{tabular}{l|ccc}
$n$       & size & $r$ & $c$ \cr
\hline
$ 10^{100}  $&$ 1.45\times 10^{76} $&$ 25 $&$ 4.345009$\cr
$ 10^{120}  $&$ 2.04\times 10^{90} $&$ 27 $&$ 4.940027$\cr
$ 10^{140}  $&$ 6.16\times 10^{113} $&$ 30 $&$ 4.037468$\cr
$ 10^{160}  $&$ 8.87\times 10^{130} $&$ 32 $&$ 4.186108$\cr
$ 10^{180}  $&$ 2.05\times 10^{149} $&$ 34 $&$ 4.169110$\cr
$ 10^{200}  $&$ 8.79\times 10^{158} $&$ 35 $&$ 5.291252$\cr
$ 10^{220}  $&$ 1.30\times 10^{179} $&$ 37 $&$ 5.024243$\cr
$ 10^{240}  $&$ 3.06\times 10^{200} $&$ 39 $&$ 4.648799$\cr
$ 10^{260}  $&$ 1.16\times 10^{223} $&$ 41 $&$ 4.175105$\cr
$ 10^{280}  $&$ 6.36\times 10^{234} $&$ 42 $&$ 4.922910$\cr
$ 10^{300}  $&$ 1.54\times 10^{259} $&$ 44 $&$ 4.294720$\cr
\end{tabular}
\caption {The value of $c$ for the Block Method}
\end{center}
\end{table}

\vfill\eject

\section{Appendix IV: Rec. Values of $d$ vs. Optimal Values}

\begin{table}[htbp]
\begin{center}
\begin{tabular}{l|rrrrrr}
 & \multicolumn{6}{c}{SPHERE} \cr
$n$ & size & $d$ & $d_{rec}$ & $k$ & $k_{rec}$ & $s$ \cr
\hline
$10^{10} $ & $2.35 \cdot 10^{6}$ &2 & 159 &11 & 4 &21\cr
$10^{11} $ & $1.13 \cdot 10^{7}$ &5 & 282 &9 & 4 &73\cr
$10^{12} $ & $7.76 \cdot 10^{7}$ &4 & 500 &10 & 4 &65\cr
$10^{13} $ & $3.91 \cdot 10^{8}$ &5 & 890 &10 & 4 &98\cr
$10^{14} $ & $2.29 \cdot 10^{9}$ &5 & 1582 &11 & 4 &102\cr
$10^{15} $ & $1.55 \cdot 10^{10}$ &6 & 2812 &11 & 4 &149\cr
$10^{16} $ & $8.57 \cdot 10^{10}$ &8 & 793 &11 & 5 &237\cr
$10^{17} $ & $5.42 \cdot 10^{11}$ &9 & 1256 &11 & 5 &314\cr
$10^{18} $ & $3.46 \cdot 10^{12}$ &12 & 1991 &11 & 5 &521\cr
$10^{19} $ & $2.35 \cdot 10^{13}$ &10 & 3155 &12 & 5 &426\cr
$10^{20} $ & $1.51 \cdot 10^{14}$ &12 & 5000 &12 & 5 &606\cr
$10^{25} $ & $2.12 \cdot 10^{18}$ &22 & 7339 &13 & 6 &2110\cr
$10^{30} $ & $3.50 \cdot 10^{22}$ &37 & 50000 &14 & 6 &6215\cr
$10^{35} $ & $6.89 \cdot 10^{26}$ &57 & 340647 &15 & 6 &15824\cr
$10^{40} $ & $1.55 \cdot 10^{31}$ &83 & 258974 &16 & 7 &35952\cr
$10^{45} $ & $3.79 \cdot 10^{35}$ &116 & 1341348 &17 & 7 &74704\cr
$10^{50} $ & $1.01 \cdot 10^{40}$ &156 & 899140 &18 & 8 &143665\cr
$10^{55} $ & $2.87 \cdot 10^{44}$ &204 & 3749472 &19 & 8 &258929\cr
$10^{60} $ & $1.69 \cdot 10^{49}$ &259 & 15811389 &20 & 8 &441294\cr
$10^{65} $ & $5.33 \cdot 10^{53}$ &322 & 8340503 &21 & 9 &715666\cr
%
%
\end{tabular}
\caption {The values of $d$, $k$, and $s$ that maximize the 3-free subsets
of $[n]$ found by the basic sphere method, along with the $d$ and $k$
recommended by formulas.}
\end{center}
\end{table}

\vfill\eject

\section{Appendix V: The Value of $c$ for the Sphere Method}

\begin{table}[htpb]
\begin{center}
\begin{tabular}{ccc@{\hspace{0.3in}}|@{\hspace{0.3in}}ccc@{\hspace{0.3in}}|@{\hspace{0.3in}}ccc}
$n$&size&$c$&$n$&size&$c$&$n$&size&$c$\cr
\hline
$10^{1}$&$2$&$1.273954$&$10^{23}$&$7.65\times10^{16}$&$2.324463$&$10^{45}$&$7.31\times10^{35}$&$2.482269$\cr
$10^{2}$&$12$&$1.186736$&$10^{24}$&$5.17\times10^{17}$&$2.338832$&$10^{46}$&$5.59\times10^{36}$&$2.486448$\cr
$10^{3}$&$42$&$1.448738$&$10^{25}$&$3.67\times10^{18}$&$2.345828$&$10^{47}$&$4.26\times10^{37}$&$2.491226$\cr
$10^{4}$&$240$&$1.476126$&$10^{27}$&$1.73\times10^{20}$&$2.371840$&$10^{48}$&$3.27\times10^{38}$&$2.495356$\cr
$10^{5}$&$736$&$1.738705$&$10^{28}$&$1.26\times10^{21}$&$2.376522$&$10^{49}$&$2.53\times10^{39}$&$2.498775$\cr
$10^{6}$&$5376$&$1.688719$&$10^{29}$&$8.90\times10^{21}$&$2.386288$&$10^{50}$&$1.96\times10^{40}$&$2.502237$\cr
$10^{7}$&$2.08\times10^4$&$1.847543$&$10^{30}$&$6.33\times10^{22}$&$2.395423$&$10^{51}$&$1.52\times10^{41}$&$2.505763$\cr
$10^{8}$&$1.08\times10^5$&$1.911638$& $10^{31}$&$4.66\times10^{23}$&$2.400014$& $10^{52}$&$1.18\times10^{42}$&$2.509345$\cr
$10^{9}$&$5.73\times10^6$&$1.969546$& $10^{32}$&$3.35\times10^{24}$&$2.408401$& $10^{53}$&$9.13\times10^{42}$&$2.513452$\cr
$10^{10}$&$2.74\times10^7$&$2.053144$& $10^{33}$&$2.40\times10^{25}$&$2.417581$& $10^{54}$&$7.15\times10^{43}$&$2.516402$\cr
$10^{11}$&$1.56\times10^7$&$2.092028$& $10^{34}$&$1.73\times10^{26}$&$2.426201$& $10^{55}$&$5.60\times10^{44}$&$2.519500$\cr
$10^{12}$&$9.81\times10^7$&$2.108959$& $10^{35}$&$1.29\times10^{27}$&$2.430556$& $10^{56}$&$4.39\times10^{45}$&$2.522653$\cr
$10^{13}$&$5.27\times10^8$&$2.162636$& $10^{36}$&$9.63\times10^{27}$&$2.435128$& $10^{57}$&$3.45\times10^{46}$&$2.525689$\cr
$10^{14}$&$3.51\times10^{9}$&$2.169946$& $10^{37}$&$7.09\times10^{28}$&$2.441841$& $10^{58}$&$2.71\times10^{47}$&$2.528914$\cr
$10^{15}$&$2.10\times10^{10}$&$2.201351$& $10^{38}$&$5.24\times10^{29}$&$2.448323$& $10^{59}$&$2.12\times10^{48}$&$2.532694$\cr
$10^{16}$&$1.33\times10^{11}$&$2.221836$& $10^{39}$&$3.91\times10^{30}$&$2.453841$& $10^{60}$&$1.69\times10^{49}$&$2.534664$\cr
$10^{17}$&$8.25\times10^{11}$&$2.247178$& $10^{40}$&$2.94\times10^{31}$&$2.458659$& $10^{61}$&$1.34\times10^{50}$&$2.537321$\cr
$10^{18}$&$5.68\times10^{12}$&$2.253504$& $10^{41}$&$2.20\times10^{32}$&$2.464334$& $10^{62}$&$1.07\times10^{51}$&$2.539395$\cr
$10^{19}$&$3.78\times10^{13}$&$2.267350$& $10^{42}$&$1.66\times10^{33}$&$2.469219$& $10^{63}$&$8.48\times10^{51}$&$2.542350$\cr
$10^{20}$&$2.39\times10^{14}$&$2.291080$& $10^{43}$&$1.26\times10^{34}$&$2.473619$& $10^{64}$&$6.73\times10^{52}$&$2.545279$\cr
$10^{21}$&$1.63\times10^{15}$&$2.301971$& $10^{44}$&$9.63\times10^{34}$&$2.477426$& $10^{65}$&$5.33\times10^{53}$&$2.548522$\cr
$10^{22}$&$1.22\times10^{16}$&$2.297940$&          &                   &          &          &                   &          \cr
\hline
\end{tabular}
\caption{The value of $c$ for the SPHERE Method}
\end{center}
\end{table}

\vfill\eject

\section{Appendix VI: Comparing Different Sphere Methods}

\begin{table}[htbp]
\begin{center}
\begin{tabular}{l|rrrr}
$n$ & SPHERE-NZ & SPHERE & SPHERE-NN & Who Wins \cr
\hline
$10$   & 4  & 2  & 5  & SPHERE-NN\cr
$100$  & 20 & 12 & 16 &  SPHERE-NZ\cr
$1000$ & 58 & 40 & 63 &  SPHERE-NN \cr
$10^4$ & 288& 240& 252 & SPHERE-NZ \cr
$10^5$ & 960 & 672 & 924 & SPHERE-NZ \cr
$10^6$ & 5376 & 5376 & 3432 &SPHERE-NZ \cr
$10^7$ & 23040 & 17600 & 12870 &  SPHERE-NZ \cr
$10^8$ & $1.07 \cdot 10^5$ & 95200 & 61894 & SPHERE-NZ \cr
$10^9$    & $5.97 \cdot 10^5$ & $4.88 \cdot 10^5$ & $3.00 \cdot 10^5$ & SPHERE-NZ\cr
$10^{10}$ & $2.89\cdot 10^6$  & $2.35 \cdot 10^6$ & $1.40 \cdot 10^6$ & SPHERE-NZ\cr
$10^{11}$ & $1.66 \cdot 10^7$ & $1.13 \cdot 10^7$ & $6.98 \cdot 10^6$ & SPHERE-NZ\cr
$10^{12}$ & $1.04 \cdot 10^8$ & $7.76 \cdot 10^7$ & $4.20 \cdot 10^7$ & SPHERE-NZ \cr
$10^{13}$ & $5.41 \cdot 10^8$ & $3.91 \cdot 10^8$ & $2.25 \cdot 10^8$ & SPHERE-NZ \cr
$10^{14}$ & $3.66 \cdot 10^9$ & $2.29 \cdot 10^9$ & $1.32 \cdot 10^9$ & SPHERE-NZ \cr
$10^{15}$ & $2.18 \cdot 10^{10}$ & $1.15 \cdot 10^{10}$ & $8.08 \cdot 10^9$ & SPHERE-NZ\cr
$10^{16}$ & $1.36 \cdot 10^{11}$ & $8.57 \cdot 10^{10}$ & $4.88 \cdot 10^{10}$ & SPHERE-NZ \cr
\end{tabular}
\caption {SPHERE-NZ vs SPHERE vs SPHERE-NN}
\end{center}
\end{table}

\vfill\eject

\section{Appendix VII: Comparing all Methods for large $n$}

\begin{table}[htbp]
\begin{center}
\begin{tabular}{ccccccc}
$                  n$ & $           {\bf B3}$ & ${\bf B5}           $ & ${\bf KD}$       &{\bf BL}& {\bf SP} &{\bf ORDER} \cr
\hline
$                 10$ & $                  5$ & $                  3$ & $            2$  &   1    & 4        & B3$>$B5$>$KD$>$SP$>$BL\cr
$                100$ & $                 24$ & $                 12$ & $            7$  &   1    &  20      & B3$>$B5$>$KD=SP$>$BL\cr
$               1000$ & $                105$ & $                 56$ & $           29$  &  1     & 58 & B3$>$B5$>$KD$>$SP$>$BL\cr
           &                        &                       &                        &                 &       & SP$>$KD !!!\cr
$  10^{4}$            & $                512$ & $                240$ & $          126$  &  2     & 288& B3$>$B5=SP$>$KD$>$BL\cr
$  10^{5}$ & $               2048$ & $                912$ & $                462$  &  8 & 960    & B3$>$B5$>$SP$>$KD$>$BL\cr
$  10^{6}$ & $               8192$ & $               5376$ & $               1716$  &  8 & 5376   & B3$>$B5=SP$>$KD$>$BL\cr
$  10^{7}$ & $ 3.28 \cdot 10^{4}$ & $ 1.72 \cdot 10^{4}$ & $               6435$  &  64 & $2.30\cdot 10^4$    & B3$>$SP$>$B5$>$KD$>$BL\cr
$  10^{8}$ & $ 1.31 \cdot 10^{5}$ & $ 9.03 \cdot 10^{4}$ & $ 2.49 \cdot 10^{4}$  &  64 & $1.07\cdot 10^5$    & B3$>$SP$>$B5$>$KD$>$BL\cr
           &                        &                       &                        &                 &       & SP$>$B3 !!!\cr
$  10^{9}$ & $ 5.24 \cdot 10^{5}$ & $ 3.66 \cdot 10^{5}$ & $ 9.24 \cdot 10^{4}$  &  $1.02\cdot 10^3$ & $5.97\cdot 10^5$    & SP$>$B3$>$B5$>$KD$>$BL\cr
$ 10^{10}$ & $ 2.10 \cdot 10^{6}$ & $ 2.05 \cdot 10^{6}$ & $ 5.05 \cdot 10^{5}$  &  $1.02\cdot 10^3$ & $2.89\cdot 10^6$    & SP$>$B3$>$B5$>$KD$>$BL\cr
$ 10^{11}$ & $ 1.05 \cdot 10^{7}$ & $ 8.95 \cdot 10^{6}$ & $ 1.70 \cdot 10^{6}$  &  $3.27\cdot 10^4$ & $1.66\cdot 10^7$    & SP$>$B3$>$B5$>$KD$>$BL\cr
$ 10^{12}$ & $ 5.03 \cdot 10^{7}$ & $ 4.17 \cdot 10^{7}$ & $ 9.42 \cdot 10^{6}$  &  $3.27\cdot 10^4$ & $1.04\cdot 10^8$    & SP$>$B3$>$B5$>$KD$>$BL\cr
$ 10^{13}$ & $ 2.01 \cdot 10^{8}$ & $ 2.22 \cdot 10^{8}$ & $ 3.99 \cdot 10^{7}$  &  $3.27\cdot 10^4$ & $5.41\cdot 10^8$    & SP$>$B3$>$B5$>$KD$>$BL\cr
$ 10^{14}$ & $ 9.77 \cdot 10^{8}$ & $ 7.41 \cdot 10^{8}$ & $ 1.61 \cdot 10^{8}$  &  $2.10\cdot 10^6$ & $3.66\cdot 10^9$    & SP$>$B3$>$B5$>$KD$>$BL\cr
$ 10^{15}$ & $ 4.29 \cdot 10^{9}$ & $ 4.27 \cdot 10^{9}$ & $ 7.03 \cdot 10^{8}$  &  $2.10\cdot 10^6$ & $2.18\cdot 10^{10}$ & SP$>$B3$>$B5$>$KD$>$BL\cr
$ 10^{16}$ & $ 1.72 \cdot 10^{10}$ & $1.61 \cdot 10^{10}$ & $ 3.16 \cdot 10^{9}$  &  $2.10\cdot 10^6$ & $1.36\cdot 10^{11}$    & SP$>$B3$>$B5$>$KD$>$BL\cr
           &                        &                       &                        &                 &       & B5$>$B3 !!!\cr
$ 10^{17}$ & $6.87 \cdot 10^{10}$ & $9.36 \cdot 10^{10}$ & $1.50 \cdot 10^{10}$  &  $2.68\cdot 10^{8}$ & $8.48\cdot 10^{11}$ & SP$>$B5$>$B3$>$KD$>$BL\cr
$ 10^{18}$ & $2.75 \cdot 10^{11}$ & $4.10 \cdot 10^{11}$ & $7.60 \cdot 10^{10}$  &  $2.68\cdot 10^{8}$ & $5.82\cdot 10^{12}$ & SP$>$B5$>$B3$>$KD$>$BL\cr
$ 10^{19}$ & $1.10 \cdot 10^{12}$ & $1.98 \cdot 10^{12}$ & $4.27 \cdot 10^{11}$  &  $2.68\cdot 10^{8}$ & $3.85\cdot 10^{13}$ & SP$>$B5$>$B3$>$KD$>$BL\cr
$ 10^{20}$ & $4.40 \cdot 10^{12}$ & $1.05 \cdot 10^{13}$ & $2.31 \cdot 10^{12}$  &  $6.87\cdot 10^{10}$ & $2.41\cdot 10^{14}$ & SP$>$B5$>$B3$>$KD$>$BL\cr
\end{tabular}
\caption{B3 vs. B5 vs. KD vs. BL vs. SP- The First Three Crossover Points}
\end{center}
\end{table}

\vfill\eject

\begin{table}[htbp]
\begin{center}
\begin{tabular}{ccccccc}
$                  n$ & $           {\bf B3}$ & ${\bf B5}           $ & ${\bf KD}$       &{\bf BL}& {\bf SP} &{\bf ORDER} \cr
$ 10^{21}$ & $1.92 \cdot 10^{13}$ & $3.84 \cdot 10^{13}$ & $5.55 \cdot 10^{12}$  &  $6.87\cdot 10^{10}$ & $1.65\cdot 10^{15}$ & SP$>$B5$>$B3$>$KD$>$BL\cr
$ 10^{22}$ & $9.57 \cdot 10^{13}$ & $2.10 \cdot 10^{14}$ & $3.49 \cdot 10^{13}$  &  $6.87\cdot 10^{10}$ & $1.13\cdot 10^{16}$ & SP$>$B5$>$B3$>$KD$>$BL\cr
$ 10^{23}$ & $4.22 \cdot 10^{14}$ & $8.12 \cdot 10^{14}$ & $2.36 \cdot 10^{14}$  &  $6.87\cdot 10^{10}$ & $7.68\cdot 10^{16}$ & SP$>$B5$>$B3$>$KD$>$BL\cr
$ 10^{24}$ & $1.97 \cdot 10^{15}$ & $4.70 \cdot 10^{15}$ & $7.41 \cdot 10^{14}$  &  $3.52\cdot 10^{13}$ & $5.21\cdot 10^{17}$ & SP$>$B5$>$B3$>$KD$>$BL\cr
$ 10^{25}$ & $9.01 \cdot 10^{15}$ & $2.10 \cdot 10^{16}$ & $4.43 \cdot 10^{15}$  &  $3.52\cdot 10^{13}$ & $3.69\cdot 10^{18}$	& SP$>$B5$>$B3$>$KD$>$BL\cr
$ 10^{26}$ & $3.60 \cdot 10^{16}$ & $1.01 \cdot 10^{17}$ & $2.49 \cdot 10^{16}$  &  $3.52\cdot 10^{13}$ & $2.47\cdot 10^{19}$ 	& SP$>$B5$>$B3$>$KD$>$BL\cr
$ 10^{27}$ & $1.44 \cdot 10^{17}$ & $5.45 \cdot 10^{17}$ & $1.04 \cdot 10^{17}$  &  $3.52\cdot 10^{13}$ & $1.76\cdot 10^{20}$ 	& SP$>$B5$>$B3$>$KD$>$BL\cr
           &                        &                       &                        &                      &   & KD$>$B3 !!!\cr
$ 10^{28}$ & $5.76 \cdot 10^{17}$ & $2.13 \cdot 10^{18}$ & $6.17 \cdot 10^{17}$  &  $3.60\cdot 10^{16}$ & $1.26\cdot 10^{21}$ & SP$>$B5$>$KD$>$B3$>$BL\cr
$ 10^{29}$ & $2.31 \cdot 10^{18}$ & $1.10 \cdot 10^{19}$ & $2.60 \cdot 10^{18}$  &  $3.60\cdot 10^{16}$ & $8.93\cdot 10^{21}$ & SP$>$B5$>$KD$>$B3$>$BL\cr
$ 10^{30}$ & $9.22 \cdot 10^{18}$ & $4.21 \cdot 10^{19}$ & $1.61 \cdot 10^{19}$  &  $3.60\cdot 10^{16}$ & $6.35\cdot 10^{22}$ & SP$>$B5$>$KD$>$B3$>$BL\cr
$ 10^{31}$ & $3.69 \cdot 10^{19}$ & $2.47 \cdot 10^{20}$ & $8.88 \cdot 10^{19}$  &  $3.60\cdot 10^{16}$ & $4.68\cdot 10^{23}$ & SP$>$B5$>$KD$>$B3$>$BL\cr
$ 10^{32}$ & $1.84 \cdot 10^{20}$ & $1.10 \cdot 10^{21}$ & $4.32 \cdot 10^{20}$  &  $7.38\cdot 10^{19}$ & $3.35\cdot 10^{24}$ & SP$>$B5$>$KD$>$B3$>$BL\cr
$ 10^{33}$ & $8.85 \cdot 10^{20}$ & $5.48 \cdot 10^{21}$ & $2.54 \cdot 10^{21}$  &  $7.38\cdot 10^{19}$ & $2.41\cdot 10^{25}$ & SP$>$B5$>$KD$>$B3$>$BL\cr
$ 10^{34}$ & $3.54 \cdot 10^{21}$ & $2.88 \cdot 10^{22}$ & $6.12 \cdot 10^{21}$  &  $7.38\cdot 10^{19}$ & $1.74\cdot 10^{26}$ & SP$>$B5$>$KD$>$B3$>$BL\cr
$ 10^{35}$ & $1.77 \cdot 10^{22}$ & $1.13 \cdot 10^{23}$ & $4.79 \cdot 10^{22}$  &  $7.38\cdot 10^{19}$ & $1.29\cdot 10^{27}$ & SP$>$B5$>$KD$>$B3$>$BL\cr
$ 10^{36}$ & $7.56 \cdot 10^{22}$ & $6.30 \cdot 10^{23}$ & $2.92 \cdot 10^{23}$  &  $7.38\cdot 10^{19}$ & $9.64\cdot 10^{27}$ & SP$>$B5$>$KD$>$B3$>$BL\cr
$ 10^{37}$ & $3.02 \cdot 10^{23}$ & $2.22 \cdot 10^{24}$ & $1.78 \cdot 10^{24}$  &  $3.02\cdot 10^{23}$ & $7.10\cdot 10^{28}$ & SP$>$B5$>$KD$>$B3$=$BL\cr
$ 10^{38}$ & $1.21 \cdot 10^{24}$ & $1.31 \cdot 10^{25}$ & $7.66 \cdot 10^{24}$  &  $3.02 \cdot 10^{23}$& $5.24\cdot 10^{29}$ & SP$>$B5$>$KD$>$B3$>$BL\cr
$ 10^{39}$ & $4.84 \cdot 10^{24}$ & $5.84 \cdot 10^{25}$ & $2.88 \cdot 10^{25}$  &  $3.02\cdot 10^{23}$ & $3.92\cdot 10^{30}$ & SP$>$B5$>$KD$>$B3$>$BL\cr
$ 10^{40}$ & $1.93 \cdot 10^{25}$ & $3.09 \cdot 10^{26}$ & $1.21 \cdot 10^{26}$  &  $3.02\cdot 10^{23}$ & $2.94\cdot 10^{31}$  & SP$>$B5$>$KD$>$B3$>$BL\cr
\end{tabular}
\caption{B3 beats KD}
\end{center}
\end{table}

\vfill\eject

\begin{table}[htbp]
\begin{center}
\begin{tabular}{ccccccc}
$ 10^{41}$ & $7.74 \cdot 10^{25}$ & $1.54 \cdot 10^{27}$ & $9.27 \cdot 10^{26}$  & $2.48\cdot 10^{27}  $& $2.21\cdot 10^{32}$ & SP$>$BL$>$B5$>$KD$>$B3\cr
$ 10^{42}$ & $3.48 \cdot 10^{26}$ & $6.15 \cdot 10^{27}$ & $5.53 \cdot 10^{27}$  &  $2.48\cdot 10^{27} $& $1.66\cdot 10^{33}$ & SP$>$B5$>$BL$>$KD$>$B3\cr
$ 10^{43}$ & $1.70 \cdot 10^{27}$ & $3.61 \cdot 10^{28}$ & $2.90 \cdot 10^{28}$  &  $2.48\cdot 10^{27} $& $1.26\cdot 10^{34}$ & SP$>$B5$>$KD$>$BL$>$B3\cr
$ 10^{44}$ & $7.43 \cdot 10^{27}$ & $1.21 \cdot 10^{29}$ & $1.41 \cdot 10^{29}$  &  $2.48\cdot 10^{27} $    & $9.63\cdot 10^{34}$   & SP$>$KD$>$B5$>$B3$>$BL\cr
$ 10^{45}$ & $3.47 \cdot 10^{28}$ & $7.15 \cdot 10^{29}$ & $8.30 \cdot 10^{29}$  &  $2.48\cdot 10^{27} $    & $7.32\cdot 10^{35}$   & SP$>$KD$>$B5$>$B3$>$BL\cr
$ 10^{46}$ & $1.58 \cdot 10^{29}$ & $3.21 \cdot 10^{30}$ & $4.59 \cdot 10^{30}$  &  $2.48\cdot 10^{27} $   &  $5.59\cdot 10^{36}$   & SP$>$KD$>$B5$>$B3$>$BL\cr
$ 10^{47}$ & $6.34 \cdot 10^{29}$ & $1.78 \cdot 10^{31}$ & $1.93 \cdot 10^{31}$  &  $4.06\cdot 10^{31} $   &  $4.27\cdot 10^{37}$   & SP$>$BL$>$KD$>$B5$>$B3\cr
$ 10^{48}$ & $2.54 \cdot 10^{30}$ & $8.46 \cdot 10^{31}$ & $1.11 \cdot 10^{32}$  &  $4.06\cdot 10^{31} $   &  $3.27\cdot 10^{38}$   & SP$>$KD$>$B5$>$BL$>$B3\cr
$ 10^{49}$ & $1.01 \cdot 10^{31}$ & $3.35 \cdot 10^{32}$ & $6.27 \cdot 10^{32}$  &  $4.06\cdot 10^{31} $   & $2.53\cdot 10^{39}$    & SP$>$KD$>$B5$>$BL$>$B3\cr
$ 10^{50}$ & $4.06 \cdot 10^{31}$ & $1.99 \cdot 10^{33}$ & $2.89 \cdot 10^{33}$  &  $4.06\cdot 10^{31} $   & $1.96\cdot 10^{40}$    & SP$>$KD$>$B5$>$BL$=$B3\cr
$ 10^{51}$ & $1.62 \cdot 10^{32}$ & $6.67 \cdot 10^{33}$ & $1.48 \cdot 10^{34}$  &  $4.06\cdot 10^{31} $   & $1.52\cdot 10^{41}$    & SP$>$KD$>$B5$>$B3$>$BL\cr
$ 10^{52}$ & $6.49 \cdot 10^{32}$ & $3.95 \cdot 10^{34}$ & $1.12 \cdot 10^{35}$  &  $1.33\cdot 10^{36} $   & $1.18\cdot 10^{42}$     & SP$>$BL$>$KD$>$B5$>$B3\cr
$ 10^{53}$ & $3.25 \cdot 10^{33}$ & $1.76 \cdot 10^{35}$ & $7.16 \cdot 10^{35}$  &  $1.33\cdot 10^{36} $   & $9.13\cdot 10^{42}$     & SP$>$BL$>$KD$>$B5$>$B3\cr
$ 10^{54}$ & $1.56 \cdot 10^{34}$ & $1.05 \cdot 10^{36}$ & $3.95 \cdot 10^{36}$  &  $1.33\cdot 10^{36} $   & $7.15\cdot 10^{43}$     &SP$>$BL$>$KD$>$B5$>$B3\cr
$ 10^{55}$ & $6.75 \cdot 10^{34}$ & $4.67 \cdot 10^{36}$ & $2.15 \cdot 10^{37}$  &  $1.33\cdot 10^{36} $   & $5.61\cdot 10^{44}$     &SP$>$BL$>$KD$>$B5$>$B3\cr
$ 10^{56}$ & $3.32 \cdot 10^{35}$ & $1.92 \cdot 10^{37}$ & $1.20 \cdot 10^{38}$  &  $1.33\cdot 10^{36} $   & $4.39\cdot 10^{45}$     &SP$>$KD$>$BL$>$B5$>$B3\cr
$ 10^{57}$ & $1.33 \cdot 10^{36}$ & $1.14 \cdot 10^{38}$ & $3.30 \cdot 10^{38}$  &  $1.33\cdot 10^{36} $   & $3.45\cdot 10^{46}$     &SP$>$KD$>$BL$>$B5$>$B3\cr
           &                      &                      &                       &                         &                        & BL$>$B5 !!!\cr
$ 10^{58}$ & $5.32 \cdot 10^{36}$ & $3.67 \cdot 10^{38}$ & $1.75 \cdot 10^{39}$  &  $8.71\cdot 10^{40} $   & $2.71\cdot 10^{47}$     &SP$>$BL$>$KD$>$B5$>$B3\cr
$ 10^{59}$ & $2.13 \cdot 10^{37}$ & $2.18 \cdot 10^{39}$ & $1.05 \cdot 10^{40}$  &  $8.71\cdot 10^{40} $   & $2.12\cdot 10^{48}$     &SP$>$BL$>$KD$>$B5$>$B3\cr
$ 10^{60}$ & $8.51 \cdot 10^{37}$ & $9.74 \cdot 10^{39}$ & $7.46 \cdot 10^{40}$  &  $8.71\cdot 10^{40} $   & $1.69\cdot 10^{49}$     &SP$>$BL$>$KD$>$B5$>$B3\cr
\end{tabular}
\caption{BL overcomes B5}
\end{center}
\end{table}

\vfill\eject

\begin{table}[htbp]
\begin{center}
\begin{tabular}{ccccccc}
$       n$ &          {\bf B3}    &  {\bf B5}            &  {\bf KD}             & {\bf BL}                & {\bf SP}     &{\bf ORDER} \cr
$ 10^{61}$ & $3.40 \cdot 10^{38}$ & $5.80 \cdot 10^{40}$ & $5.86 \cdot 10^{41}$  &  $8.71\cdot 10^{40} $   & $1.34\cdot 10^{50}$     &SP$>$KD$>$BL$>$B5$>$B3\cr
$ 10^{62}$ & $1.36 \cdot 10^{39}$ & $2.59 \cdot 10^{41}$ & $3.41 \cdot 10^{42}$  &  $8.71\cdot 10^{40} $   & $1.07\cdot 10^{51}$     &SP$>$KD$>$BL$>$B5$>$B3\cr
$ 10^{63}$ & $6.47 \cdot 10^{39}$ & $1.12 \cdot 10^{42}$ & $1.42 \cdot 10^{43}$  &  $8.71\cdot 10^{40} $   & $8.48\cdot 10^{51}$ & SP$>$KD$>$BL$>$B5$>$B3\cr
$ 10^{64}$ & $3.20 \cdot 10^{40}$ & $6.63 \cdot 10^{42}$ & $8.38 \cdot 10^{43}$  &  $1.14\cdot 10^{46} $   & $6.73\cdot 10^{52}$ & SP$>$BL$>$KD$>$B5$>$B3\cr
$ 10^{65}$ & $1.31 \cdot 10^{41}$ & $2.06 \cdot 10^{43}$ & $4.93 \cdot 10^{44}$  &  $1.14\cdot 10^{46} $   & $5.33\cdot 10^{53}$ & SP$>$BL$>$KD$>$B5$>$B3\cr
$ 10^{66}$ & $6.10 \cdot 10^{41}$ & $1.22 \cdot 10^{44}$ & $3.07 \cdot 10^{45}$  &  $1.14\cdot 10^{46} $   &  & SP$>$BL$>$KD$>$B5$>$B3\cr
$ 10^{67}$ & $2.79 \cdot 10^{42}$ & $5.48 \cdot 10^{44}$ & $1.06 \cdot 10^{46}$  &  $1.14\cdot 10^{46} $   &  & SP$>$BL$>$KD$>$B5$>$B3\cr
$ 10^{68}$ & $1.12 \cdot 10^{43}$ & $3.25 \cdot 10^{45}$ & $5.71 \cdot 10^{46}$  &  $1.14\cdot 10^{46} $   &  & SP$>$KD$>$BL$>$B5$>$B3\cr
$ 10^{69}$ & $4.46 \cdot 10^{43}$ & $1.46 \cdot 10^{46}$ & $3.42 \cdot 10^{47}$  &  $1.14\cdot 10^{46} $   &  & SP$>$KD$>$BL$>$B5$>$B3\cr
$ 10^{70}$ & $1.78 \cdot 10^{44}$ & $6.69 \cdot 10^{46}$ & $2.39 \cdot 10^{48}$  &  $2.99\cdot 10^{51} $ &  & SP$>$BL$>$KD$>$B5$>$B3\cr
$ 10^{71}$ & $7.14 \cdot 10^{44}$ & $3.88 \cdot 10^{47}$ & $1.23 \cdot 10^{49}$  &  $2.99\cdot 10^{51} $ &  & SP$>$BL$>$KD$>$B5$>$B3\cr
$ 10^{72}$ & $2.85 \cdot 10^{45}$ & $1.24 \cdot 10^{48}$ & $6.21 \cdot 10^{49}$  &  $2.99\cdot 10^{51} $ &  & SP$>$BL$>$KD$>$B5$>$B3\cr
$ 10^{73}$ & $1.16 \cdot 10^{46}$ & $7.25 \cdot 10^{48}$ & $3.63 \cdot 10^{50}$  &  $2.99\cdot 10^{51} $ &  & SP$>$BL$>$KD$>$B5$>$B3\cr
$ 10^{74}$ & $5.78 \cdot 10^{46}$ & $3.08 \cdot 10^{49}$ & $2.27 \cdot 10^{51}$  &  $2.99\cdot 10^{51} $ & & SP$>$BL$>$KD$>$B5$>$B3\cr
$ 10^{75}$ & $2.74 \cdot 10^{47}$ & $1.83 \cdot 10^{50}$ & $1.82 \cdot 10^{52}$  &  $2.99\cdot 10^{51} $ & & SP$>$KD$>$BL$>$B5$>$B3\cr
$ 10^{76}$ & $1.21 \cdot 10^{48}$ & $8.20 \cdot 10^{50}$ & $1.05 \cdot 10^{53}$  &  $2.99\cdot 10^{51} $ & & SP$>$KD$>$BL$>$B5$>$B3\cr
$ 10^{77}$ & $5.85 \cdot 10^{48}$ & $4.01 \cdot 10^{51}$ & $5.60 \cdot 10^{53}$  &  $1.57\cdot 10^{57} $ & & SP$>$BL$>$KD$>$B5$>$B3\cr
$ 10^{78}$ & $2.34 \cdot 10^{49}$ & $2.18 \cdot 10^{52}$ & $1.45 \cdot 10^{54}$  &  $1.57\cdot 10^{57} $ & & SP$>$BL$>$KD$>$B5$>$B3\cr
$ 10^{79}$ & $9.35 \cdot 10^{49}$ & $7.23 \cdot 10^{52}$ & $5.78 \cdot 10^{54}$  &  $1.57\cdot 10^{57} $ & & SP$>$BL$>$KD$>$B5$>$B3\cr
$ 10^{80}$ & $3.74 \cdot 10^{50}$ & $4.31 \cdot 10^{53}$ & $7.06 \cdot 10^{55}$  &  $1.57\cdot 10^{57} $ & & SP$>$BL$>$KD$>$B5$>$B3\cr
\end{tabular}
\caption{BL gaining on KD}
\end{center}
\end{table}

\vfill\eject

\begin{table}[htbp]
\begin{center}
\begin{tabular}{ccccccc}
$       n$ &          {\bf B3}    &  {\bf B5}            &  {\bf KD}             & {\bf BL}                & {\bf SP}     &{\bf ORDER} \cr
$ 10^{81}$ & $1.50 \cdot 10^{51}$ & $1.73 \cdot 10^{54}$ & $3.89 \cdot 10^{56}$  &  $1.57\cdot 10^{57} $ & & SP$>$BL$>$KD$>$B5$>$B3\cr
$ 10^{82}$ & $5.99 \cdot 10^{51}$ & $1.03 \cdot 10^{55}$ & $2.55 \cdot 10^{57}$  &  $1.57\cdot 10^{57} $ &  & SP$>$KD$>$BL$>$B5$>$B3\cr
$ 10^{83}$ & $2.39 \cdot 10^{52}$ & $4.62 \cdot 10^{55}$ & $1.58 \cdot 10^{58}$  &  $1.57\cdot 10^{57} $ &  & SP$>$KD$>$BL$>$B5$>$B3\cr
$ 10^{84}$ & $1.20 \cdot 10^{53}$ & $2.29 \cdot 10^{56}$ & $8.40 \cdot 10^{58}$  &  $1.65\cdot 10^{63} $&  & SP$>$BL$>$KD$>$B5$>$B3\cr
$ 10^{85}$ & $5.75 \cdot 10^{53}$ & $1.23 \cdot 10^{57}$ & $4.37 \cdot 10^{59}$  &  $1.65\cdot 10^{63} $&  & SP$>$BL$>$KD$>$B5$>$B3\cr
$ 10^{86}$ & $2.30 \cdot 10^{54}$ & $4.51 \cdot 10^{57}$ & $2.38 \cdot 10^{60}$  &  $1.65\cdot 10^{63} $&  & SP$>$BL$>$KD$>$B5$>$B3\cr
$ 10^{87}$ & $1.13 \cdot 10^{55}$ & $2.55 \cdot 10^{58}$ & $2.09 \cdot 10^{61}$  &  $1.65\cdot 10^{63} $&  & SP$>$BL$>$KD$>$B5$>$B3\cr
$ 10^{88}$ & $4.90 \cdot 10^{55}$ & $9.85 \cdot 10^{58}$ & $1.13 \cdot 10^{62}$  &  $1.65\cdot 10^{63} $&  & SP$>$BL$>$KD$>$B5$>$B3\cr
$ 10^{89}$ & $1.96 \cdot 10^{56}$ & $5.85 \cdot 10^{59}$ & $6.80 \cdot 10^{62}$  &  $1.65\cdot 10^{63} $&  & SP$>$BL$>$KD$>$B5$>$B3\cr
$ 10^{90}$ & $7.85 \cdot 10^{56}$ & $2.63 \cdot 10^{60}$ & $2.24 \cdot 10^{63}$  &  $1.65\cdot 10^{63} $&  & SP$>$KD$>$BL$>$B5$>$B3\cr
            &                      &                      &                       &                      &  & BL$>$KD !!!\cr
$ 10^{91}$ & $3.14 \cdot 10^{57}$ & $1.30 \cdot 10^{61}$ & $1.22 \cdot 10^{64}$  &  $3.45\cdot 10^{69} $&  & SP$>$BL$>$KD$>$B5$>$B3\cr
$ 10^{92}$ & $1.26 \cdot 10^{58}$ & $7.02 \cdot 10^{61}$ & $7.16 \cdot 10^{64}$  &  $3.45\cdot 10^{69} $&  & SP$>$BL$>$KD$>$B5$>$B3\cr
$ 10^{93}$ & $5.02 \cdot 10^{58}$ & $2.66 \cdot 10^{62}$ & $4.35 \cdot 10^{65}$  &  $3.45\cdot 10^{69} $&  & SP$>$BL$>$KD$>$B5$>$B3\cr
$ 10^{94}$ & $2.20 \cdot 10^{59}$ & $1.46 \cdot 10^{63}$ & $2.38 \cdot 10^{66}$  &  $3.45\cdot 10^{69} $&  & SP$>$BL$>$KD$>$B5$>$B3\cr
$ 10^{95}$ & $1.10 \cdot 10^{60}$ & $5.60 \cdot 10^{63}$ & $1.69 \cdot 10^{67}$  &  $3.45\cdot 10^{69} $&  & SP$>$BL$>$KD$>$B5$>$B3\cr
$ 10^{96}$ & $4.82 \cdot 10^{60}$ & $3.34 \cdot 10^{64}$ & $1.07 \cdot 10^{68}$  &  $3.45\cdot 10^{69} $&  & SP$>$BL$>$KD$>$B5$>$B3\cr
$ 10^{97}$ & $2.25 \cdot 10^{61}$ & $1.50 \cdot 10^{65}$ & $5.90 \cdot 10^{68}$  &  $3.45\cdot 10^{69} $&  & SP$>$BL$>$KD$>$B5$>$B3\cr
$ 10^{98}$ & $1.03 \cdot 10^{62}$ & $7.42 \cdot 10^{65}$ & $4.21 \cdot 10^{69}$  &  $1.45\cdot 10^{76} $&  & SP$>$BL$>$KD$>$B5$>$B3\cr
$ 10^{99}$ & $4.11 \cdot 10^{62}$ & $4.00 \cdot 10^{66}$ & $2.36 \cdot 10^{70}$  &  $1.45\cdot 10^{76} $&  & SP$>$BL$>$KD$>$B5$>$B3\cr
$ 10^{100}$ & $1.65 \cdot 10^{63}$ & $1.59 \cdot 10^{67}$ & $1.56 \cdot 10^{71}$  & $1.45\cdot 10^{76} $&  & SP$>$BL$>$KD$>$B5$>$B3\cr
\end{tabular}
\caption{BL beats KD}
\end{center}
\end{table}

\vfill\eject

\begin{table}[htbp]
\begin{center}
\begin{tabular}{ccccccc}
$       n$ &          {\bf B3}    &  {\bf B5}            &  {\bf KD}             & {\bf BL}                & {\bf SP}     &{\bf ORDER} \cr
$ 10^{101}$ & $6.58 \cdot 10^{63}$ & $8.70 \cdot 10^{67}$ & $8.62 \cdot 10^{71}$  & $1.45\cdot 10^{76} $&  & SP$>$BL$>$KD$>$B5$>$B3\cr
$ 10^{102}$ & $2.63 \cdot 10^{64}$ & $3.18 \cdot 10^{68}$ & $6.30 \cdot 10^{72}$  & $1.45\cdot 10^{76} $&  & SP$>$BL$>$KD$>$B5$>$B3\cr
$ 10^{103}$ & $1.05 \cdot 10^{65}$ & $1.90 \cdot 10^{69}$ & $1.80 \cdot 10^{73}$  & $1.45\cdot 10^{76} $&  & SP$>$BL$>$KD$>$B5$>$B3\cr
$ 10^{104}$ & $4.21 \cdot 10^{65}$ & $8.51 \cdot 10^{69}$ & $5.97 \cdot 10^{73}$  & $1.45\cdot 10^{76} $&  & SP$>$BL$>$KD$>$B5$>$B3\cr
$ 10^{105}$ & $2.11 \cdot 10^{66}$ & $4.51 \cdot 10^{70}$ & $4.81 \cdot 10^{74}$  & $1.45\cdot 10^{76} $&  & SP$>$BL$>$KD$>$B5$>$B3\cr
$ 10^{106}$ & $1.01 \cdot 10^{67}$ & $2.27 \cdot 10^{71}$ & $4.05 \cdot 10^{75}$  & $1.21\cdot 10^{83} $&  & SP$>$BL$>$KD$>$B5$>$B3\cr
$ 10^{107}$ & $4.04 \cdot 10^{67}$ & $9.10 \cdot 10^{71}$ & $2.83 \cdot 10^{76}$  & $1.21\cdot 10^{83} $&  & SP$>$BL$>$KD$>$B5$>$B3\cr
$ 10^{108}$ & $2.02 \cdot 10^{68}$ & $5.25 \cdot 10^{72}$ & $1.59 \cdot 10^{77}$  & $1.21\cdot 10^{83} $&  & SP$>$BL$>$KD$>$B5$>$B3\cr
$ 10^{109}$ & $8.63 \cdot 10^{68}$ & $1.82 \cdot 10^{73}$ & $9.21 \cdot 10^{77}$  & $1.21\cdot 10^{83} $&  & SP$>$BL$>$KD$>$B5$>$B3\cr
$ 10^{110}$ & $3.45 \cdot 10^{69}$ & $1.09 \cdot 10^{74}$ & $4.88 \cdot 10^{78}$  & $1.21\cdot 10^{83} $&  & SP$>$BL$>$KD$>$B5$>$B3\cr
$ 10^{111}$ & $1.38 \cdot 10^{70}$ & $4.88 \cdot 10^{74}$ & $3.01 \cdot 10^{79}$  & $1.21\cdot 10^{83} $&  & SP$>$BL$>$KD$>$B5$>$B3\cr
$ 10^{112}$ & $5.52 \cdot 10^{70}$ & $2.74 \cdot 10^{75}$ & $1.99 \cdot 10^{80}$  & $1.21\cdot 10^{83} $&  & SP$>$BL$>$KD$>$B5$>$B3\cr
$ 10^{113}$ & $2.21 \cdot 10^{71}$ & $1.31 \cdot 10^{76}$ & $1.10 \cdot 10^{81}$  & $1.21\cdot 10^{83} $&  & SP$>$BL$>$KD$>$B5$>$B3\cr
$ 10^{114}$ & $8.83 \cdot 10^{71}$ & $5.20 \cdot 10^{76}$ & $1.02 \cdot 10^{82}$  & $2.40\cdot 10^{90} $&  & SP$>$BL$>$KD$>$B5$>$B3\cr
$ 10^{115}$ & $3.98 \cdot 10^{72}$ & $3.10 \cdot 10^{77}$ & $3.75 \cdot 10^{82}$  & $2.40\cdot 10^{90} $&  & SP$>$BL$>$KD$>$B5$>$B3\cr
$ 10^{116}$ & $1.94 \cdot 10^{73}$ & $1.04 \cdot 10^{78}$ & $3.52 \cdot 10^{83}$  & $2.40\cdot 10^{90} $&  & SP$>$BL$>$KD$>$B5$>$B3\cr
$ 10^{117}$ & $8.48 \cdot 10^{73}$ & $6.23 \cdot 10^{78}$ & $1.20 \cdot 10^{84}$  & $2.40\cdot 10^{90} $&  & SP$>$BL$>$KD$>$B5$>$B3\cr
$ 10^{118}$ & $3.96 \cdot 10^{74}$ & $2.79 \cdot 10^{79}$ & $7.35 \cdot 10^{84}$  & $2.40\cdot 10^{90} $&  & SP$>$BL$>$KD$>$B5$>$B3\cr
$ 10^{119}$ & $1.81 \cdot 10^{75}$ & $1.64 \cdot 10^{80}$ & $4.71 \cdot 10^{85}$  & $2.40\cdot 10^{90} $&  & SP$>$BL$>$KD$>$B5$>$B3\cr
\end{tabular}
\caption{The order seems to have settled down}
\end{center}
\end{table}

\vfill\eject

\section{Appendix VIII: Roth's Method used Numerically}

\begin{table}[htbp]
\begin{center}
\[
\begin{array}{cccccc}
    N     & M   &  \epsilon& {\rm Roth} & {\rm Elem}   &  {\rm Improvement} \cr
\hline
  139570 & 120 & 0.001429 & 37111 & 37311 &   0.5\%\cr
  143096 & 121 & 0.002144 & 37947 & 38253 &   0.8\%\cr
  146767 & 122 & 0.002859 & 38815 & 39235 &   1.1\%\cr
  150592 & 123 & 0.003573 & 39719 & 40257 &   1.3\%\cr
  154577 & 125 & 0.004288 & 40660 & 41323 &   1.6\%\cr
  158733 & 126 & 0.005003 & 41640 & 42434 &   1.9\%\cr
  163071 & 127 & 0.005717 & 42661 & 43593 &   2.1\%\cr
  167603 & 128 & 0.006432 & 43727 & 44805 &   2.4\%\cr
  172342 & 129 & 0.007147 & 44840 & 46072 &   2.7\%\cr
  177302 & 130 & 0.007861 & 46004 & 47398 &  2.9\%\cr
  182493 & 132 & 0.008576 & 47220 & 48785 &  3.2\%\cr
  187932 & 133 & 0.009290 & 48493 & 50239 &  3.5\%\cr
  193639 & 134 & 0.010005 & 49828 & 51765 &  3.7\%\cr
  199632 & 136 & 0.010720 & 51227 & 53367 &  4.0\%\cr
  205926 & 137 & 0.011435 & 52695 & 55050 &  4.3\%\cr
  212553 & 138 & 0.012149 & 54239 & 56821 &  4.5\%\cr
  219524 & 140 & 0.012864 & 55861 & 58685 &  4.8\%\cr
  226880 & 142 & 0.013578 & 57571 & 60651 &  5.1\%\cr
  234636 & 143 & 0.014293 & 59371 & 62724 &  5.3\%\cr
  242836 & 145 & 0.015008 & 61272 & 64917 &  5.6\%\cr
  251513 & 146 & 0.015722 & 63282 & 67236 &  5.9\%\cr
  260698 & 148 & 0.016437 & 65407 & 69692 &  6.1\%\cr
  270443 & 150 & 0.017152 & 67658 & 72297 &   6.4\%\cr
  280795 & 152 & 0.017866 & 70048 & 75064 &   6.7\%\cr
  291808 & 154 & 0.018581 & 72586 & 78008 &   7.0\%\cr
\end{array}
\]
\caption {Upper Bounds for $\sz(N)$: Roth vs Elem Method $N<300,000$}
\end{center}
\end{table}

\vfill\eject
\begin{table}[htbp]
\begin{center}
\[
\begin{array}{cccccc}
    N     & M   &  \epsilon& {\rm Roth} & {\rm Elem}   &  {\rm Improvement} \cr
\hline
  303541 & 156 & 0.019296 & 75288 & 81145 &   7.2\%\cr
  316063 & 158 & 0.020010 & 78168 & 84492 &   7.5\%\cr
  329450 & 160 & 0.020725 & 81244 & 88071 &   7.8\%\cr
  343788 & 163 & 0.021440 & 84534 & 91904 &   8.0\%\cr
  359168 & 165 & 0.022154 & 88059 & 96015 &   8.3\%\cr
  375712 & 167 & 0.022869 & 91846 & 100438 &   8.6\%\cr
  393531 & 170 & 0.023584 & 95921 & 105201 &   8.8\%\cr
  412780 & 173 & 0.024298 & 100318 & 110347 &   9.1\%\cr
  433620 & 176 & 0.025013 & 105073 & 115918 &   9.4\%\cr
  456240 & 179 & 0.025728 & 110228 & 121965 &   9.6\%\cr
  480858 & 182 & 0.026442 & 115832 & 128546 &   9.9\%\cr
  507734 & 185 & 0.027157 & 121944 & 135731 &   10.2\%\cr
  537169 & 189 & 0.027872 & 128629 & 143600 &  10.4\%\cr
  569513 & 192 & 0.028586 & 135967 & 152246 &  10.7\%\cr
  605183 & 196 & 0.029301 & 144051 & 161782 &  11.0\%\cr
  644686 & 200 & 0.030015 & 152993 & 172342 &  11.2\%\cr
  688608 & 205 & 0.030730 & 162924 & 184083 &  11.5\%\cr
  737691 & 210 & 0.031445 & 174010 & 197205 &  11.8\%\cr
  792812 & 215 & 0.032159 & 186446 & 211940 &  12.0\%\cr
  855068 & 220 & 0.032874 & 200475 & 228583 &   12.3\%\cr
  925817 & 226 & 0.033589 & 216401 & 247496 &   12.6\%\cr
\end{array}
\]
\caption {Upper Bounds for $\sz(N)$: Roth vs Elem Method $300,000<N<1,000,000$}
\end{center}
\end{table}

\vfill\eject
\begin{table}[htbp]
\begin{center}
\[
\begin{array}{cccccc}
    N     & M   &  \epsilon& {\rm Roth} & {\rm Elem}   &  {\rm Improvement} \cr
\hline
  1006778 & 233 & 0.034303 & 234606 & 269139 &   12.8\%\cr
  1100129 & 240 & 0.035018 & 255573 & 294094 &   13.1\%\cr
  1208694 & 247 & 0.035733 & 279930 & 323116 &   13.4\%\cr
  1336194 & 256 & 0.036447 & 308504 & 357200 &   13.6\%\cr
  1487573 & 265 & 0.037162 & 342391 & 397668 &   13.9\%\cr
  1669609 & 275 & 0.037877 & 383097 & 446331 &  14.2\%\cr
  1891729 & 287 & 0.038591 & 432711 & 505710 &  14.4\%\cr
  2167483 & 300 & 0.039306 & 494238 & 579426 &  14.7\%\cr
  2516890 & 316 & 0.040021 & 572112 & 672832 &  15.0\%\cr
  2970752 & 334 & 0.040735 & 673156 & 794161 &  15.2\%\cr
  3578632 & 355 & 0.041450 & 808341 & 956664 &  15.5\%\cr
  4425059 & 381 & 0.042165 & 996370 & 1182937 &  15.8\%\cr
  5665145 & 414 & 0.042879 & 1271546 & 1514445 &  16.0\%\cr
  7612893 & 457 & 0.043594 & 1703279 & 2035130 &  16.3\%\cr
  10997854 & 516 & 0.044309 & 2452757 & 2940020 &  16.6\%\cr
  17911305 & 607 & 0.045023 & 3981805 & 4788171 &  16.8\%\cr
  37026906 & 774 & 0.045738 & 8204873 & 9898282 &  17.1\%\cr
  178970459 & 999 & 0.046453 & 39530554 & 47843588 & 17.4\%\cr
\end{array}
\]
\caption {Upper Bounds for $\sz(N)$: Roth vs Elem Method $300,000<N<1,000,000$}
\end{center}
\end{table}


\end{document}